\documentclass[14pt]{extarticle}
\usepackage{amssymb,amsfonts,amsthm,amsmath,bbold}
\usepackage{mathrsfs,pifont}
\usepackage[all]{xy}
\usepackage{knitting}

\usepackage[top=1.2in, bottom=1.2in, left=0.5in,
right=0.5in]{geometry}

\usepackage[font={small,sl}]{caption}
\usepackage{tikz}
\usetikzlibrary{matrix,arrows}

\newcommand{\Thetad}{{\Theta}^{\raisebox{0.5mm}{$\scriptscriptstyle \bullet$}}}
\newcommand{\cOd}{{\cO}^{\raisebox{0.5mm}{$\scriptscriptstyle \bullet$}}}
\newcommand{\cSd}{{\cS}^{\raisebox{0.5mm}{$\scriptscriptstyle \bullet$}}}

\newcommand{\Rd}{{\mathsf{R}}^{\raisebox{0.5mm}{$\scriptscriptstyle \bullet$}}}

\usepackage{hyperref}
\usepackage[backrefs,msc-links]{amsrefs}

\newcommand{\Kn}{K^{\! \tiny\textknit{U}}}
\newcommand{\Thn}{\Theta^{\! \tiny\textknit{U}}}

\newcommand{\C}{\mathbb{C}}
\newcommand{\Gm}{\mathbb{G}_\mathbf{m}}
\newcommand{\Ct}{\mathbb{C}^\times}
\newcommand{\Q}{\mathbb{Q}}
\newcommand{\Z}{\mathbb{Z}}
\newcommand{\R}{\mathbb{R}}

\newcommand{\bT}{\mathsf{T}}
\newcommand{\bA}{\mathsf{A}}

\newcommand{\bG}{\mathsf{G}}
\newcommand{\bAb}{\overline{\bA}}

\newcommand{\bB}{\mathsf{B}}

\newcommand{\bX}{\mathsf{X}}

\newcommand{\bP}{\mathbb{P}}

\newcommand{\cI}{\mathscr{I}}

\newcommand{\cL}{\mathscr{L}}
\newcommand{\cR}{\mathscr{R}}
\newcommand{\cP}{\mathscr{P}}
\newcommand{\cU}{\mathscr{U}}
\newcommand{\cQ}{\mathscr{Q}}
\newcommand{\cF}{\mathscr{F}}

\newcommand{\cN}{\mathscr{N}}
\newcommand{\cG}{\mathscr{G}}

\newcommand{\cV}{\mathscr{V}}

\newcommand{\cE}{\mathscr{E}}
\newcommand{\cA}{\mathscr{A}}
\newcommand{\cB}{\mathscr{B}}

\newcommand{\bla}{\boldsymbol{\lambda}}

\newcommand{\du}{\boldsymbol{\delta\upsilon}}

\newcommand{\bDel}{\boldsymbol{\Delta}}

\newcommand{\bphi}{\boldsymbol{\phi}}

\newcommand{\bGamma}{\boldsymbol{\Gamma}}

\newcommand{\cH}{\mathscr{H}}
\newcommand{\fC}{\mathfrak{C}}
\newcommand{\ft}{\mathfrak{t}}

\newcommand{\cS}{\mathscr{S}}

\newcommand{\bbA}{\overline{\bA}}

\newcommand{\bTb}{\overline{\bT}}

\newcommand{\vth}{\vartheta} 

\newcommand{\cO}{\mathscr{O}}
\newcommand{\Hd}{{H}^{\raisebox{0.5mm}{$\scriptscriptstyle \bullet$}}}

\newcommand{\fa}{\mathfrak{a}}

\newcommand{\fAttr}{\Attr^f}
\newcommand{\Attrc}{\Attr^{\le}}
\newcommand{\Attrl}{\Attr^{<}}

\newcommand{\cHom}{\cH\!\!\textit{om}}
\newcommand{\cs}{_{\textup{c}}}

\newcommand{\ppf}{{\circledast}}
\newcommand{\pf}{_{\ppf}}
\newcommand{\forp}{\mathsf{p}}
\newcommand{\fmo}{\mathfrak{m}_1}
\newcommand{\Stab}{\mathsf{Stab}}
\newcommand{\Attr}{\mathsf{Attr}}

\newcommand{\dd}{\triangledown}
\newcommand{\jb}{\bar \jmath}

\newcommand{\Qq}{\mathbb{q}}
\newcommand{\flx}{\lfloor x \rfloor} 
\newcommand{\bnu}{\boldsymbol \nu}
\newcommand{\Fq}{(\!(q)\!)}
\newcommand{\FqN}{(\!(q^{1/N})\!)}
\newcommand{\tor}{\textup{tors}}
\newcommand{\gen}{\textup{gen}}

\DeclareMathOperator{\Coh}{Coh}

\DeclareMathOperator{\Hom}{Hom}

\DeclareMathOperator{\Ker}{Ker}

\DeclareMathOperator{\Aut}{Aut}

\DeclareMathOperator{\Lie}{Lie}

\DeclareMathOperator{\Ell}{Ell}

\DeclareMathOperator{\chr}{char}
\DeclareMathOperator{\pt}{pt}
\DeclareMathOperator{\cochar}{cochar}
\DeclareMathOperator{\cha}{char}

\DeclareMathOperator{\rk}{rk}
\DeclareMathOperator{\Pic}{Pic}

\DeclareMathOperator{\Spec}{Spec}
\DeclareMathOperator{\Span}{span}
\DeclareMathOperator{\Proj}{Proj}

\DeclareMathOperator{\sAttr}{Attr}
\DeclareMathOperator{\supp}{supp}

\DeclareMathOperator{\diag}{diag}

\DeclareMathOperator{\codim}{codim}

\DeclareMathOperator{\weight}{weight}
\DeclareMathOperator{\conv}{conv}
\DeclareMathOperator{\Cone}{Cone}
\DeclareMathOperator{\ord}{ord}

\DeclareMathOperator{\NS}{NS}

\newcommand{\Ld}{{\Lambda}^{\!\raisebox{0.5mm}{$\scriptscriptstyle
      \bullet$}}\!}

\newtheorem{Theorem}{Theorem}

\newtheorem{Lemma}{Lemma}[section]
\newtheorem{Proposition}[Lemma]{Proposition}

\theoremstyle{definition}
\newtheorem{Definition}{Definition}

\newcommand{\Kbar}{\overline{K}}

\begin{document}

\title{Inductive construction of stable envelopes} 
\author{Andrei Okounkov} 
\date{}
\maketitle

\abstract{We revisit the construction of stable envelopes in
  equivariant elliptic cohomology \cite{ese} and give a direct inductive proof of
  their existence and uniqueness in a rather general situation. We also discuss the specialization
  of this construction to equivariant K-theory.} 

\setcounter{tocdepth}{2}
\tableofcontents

\section{Introduction}

\subsection{Stable envelopes}

\subsubsection{}

While representation theory works with linear operators between
vector spaces, the geometric representation theory works with 
correspondences. By definition a correspondence between, say, two
smooth quasiprojective algebraic varieties $\bX$ and $\bX'$ over $\C$ is an equivariant cohomology, or K-theory, or
elliptic cohomology class on the product $\bX \times \bX'$. With appropriate
properness assumptions, these can be composed, and this composition is
linear over the corresponding cohomology theory of a point.

While this setting is extremely general, there is one very particular
class of correspondences, the stable envelopes, that has been a focus of a lot of current
research, with decisive application to geometric construction of
quantum groups, including elliptic quantum groups and related
algebras, as well as to a number of core questions in enumerative
geometry and mathematical physics, see \cites{ese, AO2, MO1, Opcmi,
  SaltLake, Rio} for an introduction.

\subsubsection{}

If an algebraic torus $\bA$ acts on $\bX$ then the fixed locus $\bX' =
\bX^\bA$ is smooth and, moreover, there is a locally closed smooth
submanifold 
\begin{equation}
\Attr = \{(x,y), \lim_{a\to 0} a\cdot x = y\} \subset \bX \times
\bX^\bA\label{defattr}
\end{equation}
for any generic choice of $a\in \bA$ going to infinity of the
torus. Stable envelopes are certain canonical extension of these
attracting (also known as stable) manifolds to well-defined
correspondences between $\bX^\bA$ and $\bX$.

{} From definitions, one easily veryfies the uniqueness of stable
envelopes. Their 
existence, however, is far from obvious, which is a reason why
they are a powerful and versatile tool. 

\subsubsection{}

A direct geometric proof of existence of stable envelopes in
equivariant cohomology was given in \cite{MO1}. That line of
argument, however, is not available in equivariant K-theory and
elliptic cohomology.

A very different argument for existence of stable envelopes, which is
specific to $\bX$ being a Nakajima quiver variety, or more generally a
GIT quotient of a certain special form, was given in \cite{ese}. Since
elliptic stable envelopes are particularly important for applications,
it highly desirable to have a more general and flexible way to
construct them.

\subsubsection{}

In essense, an equivariant elliptic cohomology class on $\bX$ is a section of a
line bundle $\cS$ on the scheme $\Ell_{\textup{eq}}(\bX)$, where the
equivariance is with respect to some group which contains $\bA$ in its
center. Being an extension of \eqref{defattr} puts a numerical
constraint on $\cS$, that is, a constraint on the degree of $\cL$.
We call bundles satisfying this constraint
\emph{attractive}, see Definition \ref{d_attr}.

The main result of the paper may be informally summarized as proving
that elliptic stable envelopes exist, with a direct inductive construction, whenever there exist attractive
line bundles for $\bX$, see the following section for a precise list
of our assumptions.

Existence of attractive line bundles is a nontrivial contraint if $\rk
\bA > 1$. Not surprisingly, the most powerful application of stable envelopes,
such as geometric construction of quantum groups, require tori of rank
more than one. 

\subsubsection{}
We also give a parallel construction in equivariant K-theory and check
that the two constructions agree when the elliptic curve of the
elliptic cohomology theory degenerates to a nodal curve.

\subsection{Assumptions}

\subsubsection{} \label{s_quasiprojective_X}

Let $\bX$ be a smooth quasi-projective algebraic variety over $\C$ with an action of
a torus $\bT$.  Since $\bX$ is smooth, it follows, see e.g.\ Theorem
5.1.25 in \cite{CG}, that the action of
$\bT$ can be linearized, that is, the quasi-projective embedding
$\bX\subset \bP(\C^N)$ may be chosen $\bT$-equivariant. We fix a
subtorus $\bA \subset \bT$.

\subsubsection{}

The logic of this paper does not require any assumptions about
equivariant formality, tautological generation, or the vanishing of
the odd cohomology of $\bX$.

\subsubsection{}\label{s_assume_Attr}

We require that the union of attracting manifolds for $\bA$ is closed
in $\bX$, see Section \ref{s_Attr_closed}. 

\subsubsection{}
Stable envelopes are improved versions of attracting manifolds \eqref{defattr} for the 
subtorus $\bA \subset \bT$. If $\rk \bA > 1$, the existence of stable
envelopes puts a nontrivial constraint on the $\bA$-action, see Section
\ref{s_attr_obs}. 

One
geometrically transparent way to satisfy this constraint is to
have an $\bA$-polarization, that is, a class $T^{1/2}_\bX \in K_\bA(\bX)$
y
such
that
$$
T\bX = T^{1/2}_\bX + \left(T^{1/2}_\bX\right)^\vee \quad \textup{in $K_\bA(\bX)$}\,. 
$$
Existence of a polarization is assumed in
the definition of elliptic stable envelopes
given in \cite{ese}. Here we work with weaker assumptions.

\subsection{Attracting manifolds}

\subsubsection{}

The setup is the same as e.g.\ Section 3.2 of \cite{MO1} or Section
3.1 of \cite{ese}. Let
$$
\bX^\bA = \bigsqcup F_i
$$
be the decomposition of the $\bA$-fixed locus into components. Since $\bX$ is smooth, each $F_i$
is smooth.

\subsubsection{} \label{s_fC} 

The $\bA$-weights in the normal bundle $N_{\bX/\bX^\bA}$ form a finite subset $\{w_i\}
\subset \cha(\bA)$. The dual hyperplanes partition the vector space 
\begin{equation}
\Lie_\R \bA = \cochar \bA \otimes_\Z \R\label{LieR}
\end{equation}
into finitely many
chambers. A choice of a chamber 
$\fC$ separates  $\{w_i\}$ into those positive on $\fC$,
which we call attracting, and those negative on $\fC$, which we call
repelling. We define
$$
a \to  0_\fC  \quad \Leftrightarrow \quad
w_i(a) \to
\begin{cases}
  0\,, & \textup{$w_i$ is attacting} \,,\\
  \infty\,, & \textup{$w_i$ is repelling} \,. 
\end{cases}
$$
The point $0_\fC$ may be interpreted as a fixed point in the  toric compactification 
$\bAb \supset \bA$ defined by the fan of the chambers.

\subsubsection{} \label{s_Attr_closed} 

While the ability to vary
$\fC$ is essential in the general development of the theory, in
this paper we can fix a choice of $\fC$ once and for all.  It gives
a locally closed submanifold \eqref{defattr}. 
Its projection to two factors is a locally closed embedding and a
fibration is affine spaces, respectively, by the classical results of
\cite{BB}.

We require that the image of $\Attr$ in $\bX$ is closed, cf.\ Section
\ref{s_assume_Attr}. 

For a component $F$ of the fixed locus $\bX^\bA$, we denote
by $\sAttr(F) \subset \bX$ its attracting manifold. This is a
projection of a component of $\Attr$ to the first factor. 

\subsubsection{} \label{s_partial} 

Since the $\bA$-action on
$\bX$ is linearized, the set of components $\{F_i\}$ is partially
ordered by the containment in the closure of the attracting
manifolds, that is,
\begin{equation}
\overline{\sAttr(F_i)} \supset F_j  \Rightarrow F_i \ge F_j
\,.\label{part_order}
\end{equation}
Iterating taking closures and attracting manifolds produces
a $\bT$-invariant singular closed subvariety
\begin{equation}
\fAttr \subset \bX \times \bX^\bA\label{fAttr}
\end{equation}
 formed by the pairs 
$(x,y)$ that belong to a chain of closures of attracting
$\bA$-orbits.
By construction, $\fAttr \subset \Attrc$, where 
\begin{equation}
\Attrc = \bigcup_{F_j \le F_i} \sAttr(F_j) \times F_i \,.\label{Attrc}
\end{equation}

\subsubsection{} 

Stable envelopes are certain canonical $\bT$-equivariant
correspondences supported on $\Attr^f$.  Their exact flavor
depends on the chosen cohomology theory. 
The main focus in this
paper is on stable envelopes in equivariant elliptic cohomology, as
defined
in \cite{ese}. 

\subsubsection{} 

The main aspect in which elliptic cohomology differs from equivariant
cohomology $\Hd_\bT(\bX)$
or equivariant
K-theory $K_\bT(\bX)$ is the
following. While cohomology classes  form a
supercommutative ring,
one doesn't tend to think about them as functions on $\Spec
\Hd_\bT(\bX)$. In elliptic theory, 
$ \Spec \Hd_\bT(\bX)$ is promoted to a superscheme 
$$
( \Spec \Hd_\bT(\bX), \Hd_\bT(\bX)) \rightsquigarrow (\Ell_\bT(\bX),
\cOd_{\Ell_\bT(\bX)})\,,
$$
which is no longer affine. As a result, global sections 
of $\cOd_{\Ell_\bT(\bX)}$ are not rich enough to account for the geometry
of $\bT$ and $\bX$. 

In a nutshell, 
elliptic cohomology classes are sections of 
line bundles on $\Ell_\bT(\bX)$. In particular, the elliptic
stable envelopes are global sections of certain line bundles on $\Ell_\bT(\bX \times
\bX^\bA)$. We note that the very existence of the required line bundles
puts a nontrivial constraint on the $\bA$-action on $\bX$, see Section
\ref{s_attr_obs}.


\subsection{Equivariant elliptic cohomology}

\subsubsection{} 

Let $E$ be an elliptic curve over a Noetherian affine base scheme
$\bB$.
For the purposes of this
paper, one doesn't loose or gain much if one assumes that
$\bB=\Spec \C$ or $\bB=\Spec \C(\!(q)\!)$. Since elliptic stable envelopes are
unique without invoking any equivalence relations, their construction
is local over the base scheme.  

\subsubsection{}\label{s_funct} 


Let $\bG$ be a compact Lie group. Equivariant elliptic cohomology, developed in 
\cites{Groj,GKV,Rosu,Lurie,Gepner,Ganter} and other papers, defines a
functor
\begin{equation}
\Ell_\bG(\bX,\partial \bX) : 
\big\{ \textup{pairs of $\bG$-spaces} \big\} \to \{\textup{graded superschemes
  over $\bB$}\} \,,\label{EllbG}
\end{equation}
covariant with respect to action-preserving maps
\begin{equation}
f: (\bG_1, \bX_1,\partial \bX_1) \to (\bG_2, \bX_2,\partial \bX_2)  \, \label{map_pairs} \,. 
\end{equation}
In this paper, we stay entirely in the world of unitary and
abelian groups $\bG$. We assume that $E$ is such that the functor
\eqref{EllbG} has been defined.

For $\bG$ connected, we denote the split reductive group over $\Z$ corresponding to $\bG$ 
by the same symbol $\bG$. With this convention, we can talk about both
the $\bA$-attracting manifolds and $\bA$-equivariant elliptic
cohomology without overloading the notation.

We abbreviate $\Ell_\bG(\bX,\partial \bX)$ to $\Ell_\bG(\bX)$ when $\partial
\bX = \varnothing$. 

\subsubsection{}

To save on notation for functorial maps, we abbreviate $\Ell(f)$ to
just $f$ in what follows. We use $f^*$ and $f_*$ to denote pullback of coherent
sheaves under $\Ell(f)$.
Note the crucial difference between $f_*$ and push-forward in elliptic
cohomology, see Section \ref{s_push}. 



\subsubsection{}

The grading in \eqref{EllbG} refers to the $\Z$-grading of the
structure sheaf
\begin{equation}
\cOd_{\Ell_\bG(\bX)} = \bigoplus_{k\in \Z}
\cO^k_{\Ell_\bG(\bX)}\label{cOd}
\end{equation}
of $\Ell_\bT(\bX)$. Periodicity in elliptic cohomology means that
$$
\cO^{k-2}_{\Ell_\bG(\bX)} = \cO^{k}_{\Ell_\bG(\bX)} \otimes \omega \,,
\quad  
\omega= T^*_\textup{origin} E  \,, 
$$
where  $\omega$ is a line bundle pulled back from the base $\bB$. This
can be interpreted as having the whole theory not over $\bB$ but over the
total space of a $\Gm$-bundle associated to $\omega$.

\subsubsection{}

Line bundles $\cL$ on $\Ell_{\bG}(\bX)$ will be, by definition, graded,
that is, we set 
\begin{equation}
\Pic(\Ell_{\bG}(\bX)) = H^1\left(
  \left(\cO^{0}_{\Ell_\bG(\bX)} \right)^\times \right) \,.\label{defPic}
\end{equation}
These can be interpreted as invertible graded 
$\cOd_{\Ell_\bG(\bX)}$-bimodules and they
form a commutative group with respect to $\otimes$.

\subsubsection{} 

Stable envelopes are even classes and $\cO^{k}_{\Ell_\bG(\bX)}$ with
$k\ne 0$ will not play an important role in this paper. For brevity,
we will suppress the degree grading from our notation, except where it
essential 
(which happens 
in the analysis of the long exact sequence in Section \ref{s_proof1}).

\subsubsection{}

By construction, 
$$
\Ell_{\Ct}(\pt) = E 
$$
and
\begin{equation}
  \label{groupl}
  \xymatrix{
    \Ct \times \Ct \ar[rrr]^{ (z_1,z_2) \to z_1z_2}
    \ar[d]_{\Ell_{\bullet}(\pt)}&&&   \Ct
    \ar[d]^{\Ell_{\bullet}(\pt)}\\
    E \times E \ar[rrr]^{\textup{group law}}&&& E
  } \,. 
\end{equation}
One can treat the coordinate on $\Ct$ is a stand-in for the unavailable coordinate on
$E$. If fact, for $\bB=\Spec\C$ one can take $E=\Ct/q^\Z$ for some
$|q|<1$.

\subsubsection{}
The map to the point
\begin{equation}
p_\bX: (\bT,\bX) \to (\bT, \pt) \label{p_X}
\end{equation}
makes $\Ell_\bT(\bX)$ a scheme over 
\begin{equation}
\cE_\bT \overset{\tiny\textup{def}} = \Ell_\bT(\pt) =
\cochar(\bT) \otimes_\Z E\,. \label{cEbT}
\end{equation}

\subsubsection{}

We will be using various constructions in elliptic cohomology which
are reviewed in Appendix \ref{s_constr}.

\subsection{Plan of the paper and acknowledgments}

\subsubsection{}
The three main sections of the paper are devoted to: (1) the proof of
existence and uniqueness of stable envelopes in equivariant elliptic
cohomology, (2) same in equivariant K-theory, (3) the relation between
the two construction.

Elliptic cohomology meets K-theory when the underlying elliptic curve
degenerates to a nodal curve of genus $1$. We prove that this
degeneration respects stable envelopes, strengthening earlier results of
\cite{ese,KonSmi}. Perhaps the reader will find some independent use for the notions of compactified
K-theory and nodal K-theory which we develop in the course of the
proof.

\subsubsection{}
Further applications of the inductive construction of stable envelopes
will be given in \cite{nonab}.

\subsubsection{}
The present paper grew our of the author's joint projects \cite{ese, DHLMO} with Mina
Aganagic, Davesh Maulik, and Daniel Halpern-Leistner. It owes a lot to
all of them. While, perhaps, we achieve a certain progress on a number
of technical points in this paper, the reader should consult
\cite{ese, DHLMO} as well as perhaps \cite{MO1, Opcmi, AO2} for a
comprehensive discussion of stable envelopes and their many
applications.

\subsubsection{}
I am very grateful to  V.~Alexeev, R.~Bezrukavnikov, B.~Bhatt, A.~Blumberg, J.~de
Jong, I.~Krichever, M.~Mustata, B.~Poonen, 
E.~Rains, R.~Rouquier, D.~Sinha, and others for valuable correspondence during
the writing of this paper.

\subsubsection{}
I am grateful to the Simons Foundation for being supported as a Simons
Investigator. I thank the
Russian Science Foundation for the support by the grant  19-11-00275.

\subsection{Dedication}

This paper was written in the summer of 2020, the time of great grief
and loss for millions of people around the globe. I would like to
dedicate it to the memory of Boris Dubrovin, whose untimely passing
back in March 2019 was such a great loss for the mathematics as a whole and for me personally.

I grew up reading his \emph{Modern Geometry}, and I cherish the
memories of our, regrettably, infrequent interactions later in
life. He always radiated enthusiasm for geometry, mathematics, music
(I might have met him more often in the Moscow Conservatory than at the
Moscow State University), and life and general. His pioneering vision
put enumerative geometry in the front and center of modern geometry
and mathematical physics. I wish I could explain the results of this
paper to him.

\section{Elliptic stable envelopes}

\subsection{Attracting manifolds again}

\subsubsection{}
Let $F$ be a component of the fixed locus
$\bX^\bA$ and consider the
diagram 
\begin{equation}
  \label{corrM}
  \xymatrix{
   \bX \times F \ar[d]_{\forp_1} \ar[drr]^{\forp_2} && \sAttr(F)
   \ar[ll]_{\jmath} \ar[d]^\pi\\
   \bX   && F \ar[ll]_\iota\,,  
    } 
  \end{equation}
  in which $\iota$ and $\jmath$ are inclusions,  $\forp_1$ and
  $\forp_2$ are projections as in \eqref{diagcorr}, and
  $$
  \pi = \forp_2 \, \jmath
  $$
  is a fibration in affine spaces, and thus a homotopy
  equivalence. In particular,
$$
  \pi: \Ell_\bT(\sAttr(F)) \xrightarrow{\quad \sim\quad} \Ell_\bT(F)
  \,, 
  $$
is an isomorphism. 

\subsubsection{} 

  Restricted to $F$, any $\bA$-equivariant K-theory class decomposes
according to the characters of $\bA$ and, in particular, splits into
attracting, repelling, and $\bA$-fixed directions. For instance, we have
\begin{equation}
  \label{TXsplit}
  T\bX\big|_{F} = N_{\bX/F,>0} + N_{\bX/F,<0} +T F \,, 
\end{equation}
where the subscripts $>0$ and $<0$ indicate the attracting and
repelling directions, respectively. With this notation, we have
\begin{equation}
 N_\jmath = \pi^*( TF + N_{\bX/F,<0})  \,,\label{Nj}
 \end{equation}
 where $N_j$ is the normal bundle to the inclusion $\jmath$ in
 \eqref{corrM}. 

 \subsubsection{}

 Recall we think of  sections of line bundles $\cL$ on $\Ell_\bT(\bX)$ as
 elliptic cohomology classes assigned to cycles in $X$. For instance,
 if we have a proper complex oriented map
 $$
 f: Y \to X 
 $$
 and there is a bundle $V$ on $X$ such
 that $f^*V = N_{f}$ then $f\pf$ induces a section
 $$
 \cO_{\cE_\bT} \to \Theta(V) 
 $$ 
 which represents in elliptic cohomology what we would call $f_*[Y]
 \in \Hd_\bT(\bX)$. For example, the inclusion $\pt \to \bX$ of a
 fixed point gives a section of
 $\cL=\Theta(T\bX)$, while the constant section of $\cL=\cO_{\Ell_\bT(\bX)}$
 corresponds to the identity map $\bX \to \bX$\,. 

 The degree
 $$
 \deg \cL \in \Pic(\Ell_\bT(\bX)) \big/ \Pic_0(\Ell_\bT(\bX))
 $$
 gives a measure of the codimension of the corresponding cycle. In
 particular, the degree $\deg_\bA \cL$ in equivariant variables
 as defined in \eqref{degbA} is a coarse measure of the
 codimension. 

 \subsubsection{}
 Since we are looking for a line bundle $\cS$ to represent attracting
 manifolds, it is logical to make the following

 \begin{Definition}\label{d_attr}
  A line bundle $\cS$ on $\Ell_\bT(\bX)$ is called \emph{attractive} for a
  given choice $\fC$ of attracting directions if
  \begin{equation}
  \deg_\bA \cS = \deg_\bA \Theta(N_{\bX/\bX^\bA,<0}) \,.
\label{dattr}
\end{equation}%
\end{Definition}

\medskip 

\noindent 
For a given line bundle $\cS$ on $\Ell_\bT(\bX)$ we define
\begin{equation}
\cS_\bA = \iota^* \cS \otimes \Theta(-N_{\bX/\bX^\bA,<0})
\in \Pic (\Ell_\bT(\bX^\bA)) \,. \label{cSbA}
\end{equation}
Clearly, \eqref{dattr} is equivalent to $\deg_\bA \cS_\bA = 0$.
Also, all attractive bundles form a principal homogenous space under
$$
\Ker \deg_\bA \subset \Pic(\Ell_\bT(\bX)) \,,
$$
which contains all K\"ahler line bundles. 

\subsubsection{}
\begin{Lemma}
  If $\cS$ is attractive and $\jmath$ in \eqref{corrM} is proper,
  then $\jmath\pf$ gives a section
  \begin{equation}
  [\Attr] : \cO_{\cE_\bT}  \to \cS \boxtimes \left(\cS_\bA\right)^\dd
  \,. 
\label{[Attr]}
\end{equation} %
\end{Lemma}

\noindent
Recall that $\boxtimes$ denotes the tensor product of pullbacks via two projection
maps, here $\forp_1$ and $\forp_2$. 
Also note that the duality
$(\, \cdot \,)^\dd$, which was defined in \eqref{dd}, is applied here
in the $\bX^\bA$ factor.

\begin{proof}
Note that  the outer square in \eqref{corrM} commutes up to
  homotopy
  \begin{equation}
    \label{homot}
    \forp_1 \, \jmath \sim \iota \, \pi \,. 
  \end{equation}
  Therefore
  $$
  \jmath^* \forp_1^* \cS = \pi^* \iota^* \cS
  $$
  and 
  $$
  \jmath^* \left(\cS \boxtimes \left(\cS_\bA\right)^\dd \right)=
  \Theta(\pi^*(T F + N_{X/F})) = \Theta(N_{\jmath}) \,,
  $$
  as required. 
\end{proof}

\subsubsection{}
{}From \eqref{TXsplit}, we observe the following

\begin{Lemma}
If $\cS$ is attractive for $\fC$ then $\cS^\dd$ is attractive for
$-\fC$. 
\end{Lemma}

\subsubsection{}
It is an interesting question to characterize $\bA$-actions having
attractive line bundles. In general, the functor
\begin{equation}
(\bG,\bX) \to \Hd\left(
  \left(\cO^{0}_{\Ell_\bG(\bX)} \right)^\times \right)
\,.\label{defPic2}
\end{equation}
that generalizes \eqref{defPic} is an interesting functor to complexes
of abelian groups.

\subsubsection{}\label{s_attr_obs} 
Simple examples, starting with the the maximal torus
$\bA \subset PGL(3)$ acting on $\bP^2$, show that attracting line bundles
$\cS$ typically do not exist if $\rk \bA > 1$. In fact,
any subtorus $\bA' \subset \bA$ gives the following potential
obstruction to the existence of attracting bundles in the style of
\cite{GKM}. 

\begin{Proposition}
Let $F_1$ and $F_2$ be two
components of $\bX^\bA$ that belong to the same component of
$X^{\bA'}$ where $\bA' \subset \bA$ is a subtorus. If
$$
\deg_{\bA/\bA'} \Theta(N_{\bX/F_1,<0}) \ne
\deg_{\bA/\bA'} \Theta(N_{\bX/F_2,<0})
$$
then no attractive bundle $\cS$ exists. 
\end{Proposition}

\noindent 
A simple way to guarantee the existence of $\cS$ is to assume that
$X$ has a polarization.

\subsection{Polarization and dynamical shifts}

\subsubsection{} 

By definition, a \emph{polarization} of $\bX$ with respect to $\bA$ is a class 
$T^{1/2}_\bX\in K_\bA(\bX)$ such that
\begin{equation}
T\bX = T^{1/2}_\bX + \left(T^{1/2}_\bX\right)^\vee\label{polar}
\end{equation}
in $K_\bA(\bX)$.  The existence of a polarization implies,
in particular, that $\dim \bX$ is even. Like any $\bA$-equivariant
K-theory class, $T^{1/2}_\bX$ may be lifted to a $\bT$-equivariant
class so that \eqref{polar} holds modulo
$$
\cI_\bA  = \textup{ideal of $\bA$} \subset K_\bT(\pt) \,.
$$ 

\subsubsection{Example} \label{ex1} 
 Suppose we  have a $\bT$-equivariant diagram 
 \begin{equation}
\xymatrix{
  \bX \ar[rr]^{\textup{open}} \ar[rrd]_{p}&& T^*M  \ar[d]\\
  && M}  \,. \label{TM}
\end{equation}
Then either $\Ker(dp)$ or $p^*TM$ give a polarization with
respect to
$$
\bA = \bT \cap \Aut(\bX,\omega) \,, 
$$
where $\omega$ is the canonical symplectic form.

The base $M$ in \eqref{TM}  may be a
quotient
stack, and Nakajima quiver
varieties \cite{Nak1} are constructed as $\bX$ of this form. Recall that
Nakajima varieties and closely related algebraic varieties are among
the most important objects in geometric representation theory. In
fact, stable envelopes give one geometric approach to extracting 
representation theory from them, see \cites{MO1, Opcmi} for an
introduction. 

\subsubsection{}

\begin{Proposition}
 If $T^{1/2}$ is a polarization then $\cS=\Theta(T^{1/2})$ is
 attractive for any $\fC$. 
\end{Proposition}

\begin{proof}
  {}From \eqref{polar} and \eqref{TXsplit} we conclude 
\begin{equation}
  \label{TXs}
 N_{\bX/\bX^\bA,<0}   = T^{1/2}_{<0} +
 \left(T^{1/2}_{>0}\right)^\vee+\du \,, \quad \du \in \cI_\bA \, K_\bT(\bX^\bA) \,, 
\end{equation}
where $T^{1/2}_{\gtrless 0}$ denote the attracting and repelling
parts of $T^{1/2}\big|_{\bX^\bA}$.
The equality \eqref{ThVdual} and
$$
\deg_\bA \Theta(\du) = 0
$$
prove the claim. 
\end{proof}

\subsubsection{}

The K-theory class $\du$ depends on the choice of $\fC$, that is, on
the choice of the attracting directions. In this sense it is
dynamic. It also corresponds to the dynamical variables in the
elliptic quantum groups, see \cite{ese}. Dynamical is a Greek word which
starts with $\delta$ and $\upsilon$, the notation $\du$ is chosen
to reflect this.

\subsubsection{}

When one starts composing elliptic stable envelopes, then the 
following formula is useful: 
\begin{equation}
\cS = \Theta(T^{1/2}_\bX) \quad \Rightarrow
\quad
\cS_\bA = \Theta(T^{1/2}_{\bX^\bA}) \otimes \Theta(-\du)
\,. \label{cSbAdu}
\end{equation}
Thus the twist
by $\Theta(\du)$ becomes the dynamical shift which is an important
phenomenon in the theory of elliptic quantum group. To set it up
correctly, it is useful to have the following
formula for $\Theta(\du)$ in terms of the K\"ahler line bundles. 

Let $\{t_i\}\subset \cha(\bT/\bA)$ be a set of coordinates on
$\bT/\bA$. By hypothesis, there exist nonunique $\upsilon_i \in
K_\bT(\bX^\bA)$ such that 
$$
\du = \sum (t_i-1) \upsilon_i \,. 
$$
Therefore
\begin{equation}
  \label{Thdyn}
  \Theta(\du) = \Theta(-{\textstyle \sum} \, t_i \rk \upsilon_i) \otimes
\bigotimes \cU(\upsilon_i,  t_i) \,.
\end{equation}
For our purposes in this
paper it is not important to unpack the bundle $\Theta(\du)$. 

\subsubsection{Example}\label{ex_2} 
Let $\bX=T^* \bP(\C^n)$ with the action of $GL(n) \times \Ct$, where
$\hbar\in \Ct$  scales the cotangent fibers with weight $\hbar^{-1}$.
Let
$$
\bA = \diag(a_1,\dots,a_n) \subset GL(n)\,, 
$$
be a maximal torus and choose $\fC$ so that
$$
a\to 0_\fC \Leftrightarrow \forall i, \,  a_{i}/a_{i+1} \to 0 \,. 
$$
We take the pullback of $T_{\bP(\C^n)}$ as the polarization of $\bX$.

The fixed locus $\bX^\bA=\sqcup F_k$ consists of $n$ isolated points --- coordinate
lines in $\C^n$. They are ordered as follows
$$
F_1 < F_2 < \dots < F_n
$$
in the sense that each $F_i$ lies in the closure of the attracting
manifold of $F_{i+1}$.

The restriction of the tangent bundle to $k$-th point
has the character
$$
T\bX\big|_{F_k} = \sum_{i\ne k}  \left ( \frac{a_i}{a_k} +
  \frac{a_k}{\hbar a_i} \right) \,, 
$$
in which the terms without $\hbar$ give the polarization. Therefore
$$
\du \big|_{F_k} = (\hbar^{-1}-1) \sum_{i < k}   \frac{a_k}{a_i} \,. 
$$
Other choices of $\fC$ will give a different order $>_{\fC}$ on the set
of fixed components $F_k$, and the general formula is
$$
\du_\fC \big|_{F_k} = (\hbar^{-1}-1) \sum_{F_i <_{\fC} F_k}   \frac{a_k}{a_i} \,. 
$$


\subsection{Resonant locus}
\label{sec:resonant-locus}

\subsubsection{}

Elliptic stable envelopes are sections of the form \eqref{[Attr]},
except they have poles for certain resonant values of the
parameters. The resonant locus $\bDel$ is defined as follows.

\subsubsection{} 

Recall that we have maps 
\begin{equation}
\Ell_\bT(\bX) \xleftarrow{\,\,\, \iota \,\, }
\Ell_\bT(\bX^\bA) 
\xrightarrow{\,\, p  \,\,\, }\cE_\bT \xrightarrow{\,\, \phi  \,\,\, }
\cE_{\bT/\bA}
\,,\label{ipphi}
\end{equation}
which correspond to
\begin{align*}
  \label{eq:9}
  \iota &= \textup{inclusion $\bX^\bA \to \bX$}\,,\\
  p &= \textup{map to a point $\bX^\bA \to \pt$}\,,\\
  \phi &= \textup{quotient $\bT \to \bT/\bA$} \,. 
\end{align*}

\subsubsection{}
Let $F_j < F_i$ be a pair of components of $\bX^\bA$ that are comparable in the
partial order from Section \ref{s_partial}
and consider the following special case of \eqref{pullA}
\begin{equation}
  \label{pullA2}
  \xymatrix{
    \Ell_\bT(F_j \times F_i) \ar[rr]^{\phi} \ar[rrd]^\psi 
    \ar[d]_{p} && \Ell_{\bT/\bA}(F_j \times F_i)
    \ar[d]^{p} \\
    \cE_\bT \ar[rr]^{\phi} && \cE_{\bT/\bA}
    }\,, 
  \end{equation}
  in which $\psi$ is the diagonal is the commuting square. 

\subsubsection{}
Let $\cS$ be an attractive line bundle on $\bX$ and recall that in
\eqref{cSbA} we have defined a line bundle $\cS_\bA$ on $X^\bA$ with
$\deg_\bA \cS_\bA = 0$. 

\begin{Definition}
  The resonant locus $\bDel$ is the union of
  $$
  p \left(\supp \Rd \phi_*  \left(\cS_{\bA,F_j} \boxtimes \left(
      \cS_{\bA,F_i} \right)^{-1}  \right) \right) \subset \cE_{\bT/\bA}
  $$
  over all pairs $F_j < F_i$  of components of the fixed locus. The
  complement of $\bDel$ is called the nonresonant set. We use the same
  terms for the preimages of $\bDel$ and its complement in
  $\Ell_\bT(\bX)$. 
\end{Definition}

\begin{Definition}
  An attractive line bundle is called nondegenerate if the nonresonant
  set is open and dense in $\Ell_\bT(\bX)$. 
\end{Definition}

\noindent 
Evidently,
$$
\supp \Rd \psi_*  \left(\cS_{\bA,F_j} \boxtimes \left(
    \cS_{\bA,F_i} \right)^{-1}  \right)  \subset \bDel \,.
$$

\subsubsection{}
Let $\cG$ be a coherent sheaf on $\Ell_\bT(\bX)$ and let
  $$
  i_{\bDel} : \Ell_\bT(\bX)_{\textup{nonresonant}} \to \Ell_\bT(\bX)
  $$
  be the inclusion of the nonresonant set. We define
  $$
  \cG(\infty \bDel) =  i_{\bDel,*} \,  i_{\bDel}^* \, \cG \,.
  $$
  Informally, these are sections of $\cG$ with poles of arbitrary
  order along $\bDel$. 
As we will see, elliptic stable envelopes will have such poles.

\subsubsection{}
Note that $\phi$ is an $\cE_\bA$-fibration and that the line bundle 
  $\cS_{\bA,F_j} \boxtimes \left(
      \cS_{\bA,F_i} \right)^{-1}$  has degree $0$ along the fibers of $\phi$. Therefore, the
following general statement can be used to bound $\bDel$.

Abstractly, let $\bphi: \cA \to \cB$ be a Abelian variety over a base
scheme $\cB$ and let $\cL$ be a line bundle on $\cA$ which is algebraically
equivalent to zero on fibers of $\bphi$. 

If $\cB=\Spec \Bbbk$, where $\Bbbk$ is a field, then
\begin{equation}
\Rd \bphi_* \cL = 0  \Leftrightarrow \cL \ne \cO_\cA \,.\label{vanishA}
\end{equation}
By semicontinuity of cohomology, this implies the following

\begin{Lemma}
If $\cL \ne \cO_\cA$ over the generic point of 
every component of $\cB$ then
$$
\cB \setminus \supp \Rd \bphi_* \cL \subset \cB
$$
is an open dense subset. 
\end{Lemma}

\subsubsection{Example}\label{ex_3}
  In the situation of Example \ref{ex_2}, let $\cO(1) \in \Pic_\bT(\bX)$
  have the standard linearization, with respect to which 
  $$
  \weight \cO(1) \big|_{F_k} = a_k^{-1} \,.
  $$
  We enlarge the base $\bB$ by pullback via 
  \begin{equation}
    \bB_{\textup{new}} = \bB_{\textup{old}} \times E_z \to
    \bB_{\textup{old}}  \label{newS} 
  \end{equation}
  and take
  $$
  \cS= \Theta(T^{1/2})\otimes \cU(\cO(1),z)
  $$
  where $\cU$ is defined in  \eqref{defcU}. We have 
  \begin{align*}
   \cU(\cO(1),z)\big|_{\Ell_\bT(F_k)} &= \Theta((z-1)(a_k^{-1}-1)) \,,
    \\
    \Theta(-\du) \big|_{\Ell_\bT(F_k)} &=
  \Theta(-(\hbar^{-1}-1) (a^{\mu_k}-1)) 
  \, ,
  \end{align*}
  where
  $$
  \mu_k = ( \underbrace{-1, \dots, -1}_\textup{$k-1$ terms}, k-1, 0,
  \dots) \,. 
  $$
  Therefore, for $i>j$, 
  $$
  \cS_{\bA,F_j} \boxtimes \left(
      \cS_{\bA,F_i} \right)^{-1}  = \Theta((z-1)(a_j/a_i-1)-(\hbar^{-1}-1)  (a^{\mu_j-\mu_i}-1)
 \,.
  $$
This has a chance to be trivial along the $\cE_\bA$-fibers only if the
characters $a_j/a_i$ and $a^{\mu_j-\mu_i}$ are dependent in
$\cha(\bA)$, which happens for 
$j=i-1$ with
$$
a^{\mu_{i-1}-\mu_i} =(a_{i-1}/a_i)^{i-1} \,.
$$
Then we get 
 $$
 \cS_{\bA,F_{i-1}} \boxtimes \left(
      \cS_{\bA,F_i} \right)^{-1} = \Theta((z h^{1-i}-1) (a_{i-1}/a_i -1)) \,.
 $$
 Thus
 $$
 \bDel= \Theta\left(\sum_{i=1}^n z \hbar^{1-i}\right) \,. 
 $$
 In English, this is the locus where
 $$
 z\in \{1,\hbar,\dots, \hbar^{n-1}\}
 $$
 if we think of these as coordinates on $\cE_{\bT/\bA} \times E_z$. 
 Compare this with the poles in $z$ in the explicit formula for
 elliptic stable envelopes for $\bX=T^* \bP(\C^n)$ discussed in
 Section 3.4 of \cite{ese}.

\subsubsection{}

\begin{Proposition}\label{p_newS} 
Every line bundle $\cS$ can be made nondegenerate if one enlarges the
base as in \eqref{newS}. 
\end{Proposition}

\begin{proof}
  Let $\cO(1)\in \Pic_\bT(\bX)$ be an ample line bundle. Take the
  pullback of $E$ to the new base \eqref{newS}  and consider
  $$
  \cS_\textup{new} = \cS_\textup{old} \otimes \cU(\cO(1), z) \,. 
  $$
  This reduces to old line bundle over the origin of the 
  $E_z$-factor.

 By construction, this twists $\cS_{\bA,F_{i-1}} \boxtimes \left(
      \cS_{\bA,F_i} \right)^{-1}$ by 
 \begin{equation}
  \label{Rpn}
   \cU\left(\cO(1)\big|_{F_j} - \cO(1)\big|_{F_i} ,z\right) \,. 
 \end{equation}
In Lemma \ref{l_weight} below, we check that $\cO(1)$ has different $\bA$-weights
  at $F_i$ and $F_j$ provided $F_j < F_i$. Therefore the line bundle \eqref{Rpn}
  is nonconstant along the $E_z$-direction and, in particular,
  nontrivial over the generic point of every component of
  $$
  \Ell_\bT(\bX)_\textup{new} =  \Ell_\bT(\bX) \times_\bB
  \bB_\textup{new}
  \,.
  $$
  \end{proof}

\subsubsection{}

\begin{Lemma}\label{l_weight} 
Suppose $F_1$ and $F_2$ are two components of the fixed locus such
that $F_1 > F_2$ and suppose $\cL \in \Pic_\bA(\bX)$ is ample. Then
\begin{equation}
\textup{weight} \, \cL \big|_{F_1} - \textup{weight} \, \cL\big|_{F_2}
>_{\fC} 0\label{weightL}
\end{equation}
where $>_{\fC} 0$ means that it is positive on the 
interior of $\fC$ as a linear function on \eqref{LieR}. 
\end{Lemma}

\begin{proof}
 Let 
$$
\sigma: \Ct \to \bA
$$
be a generic cocharacter in the interior of $\fC$. 
Since $F_i > F_j$, these two components are connected by a chain of closures of
attracting $\sigma$-orbits. Computing the $\cL$-degree of these orbits
by $\sigma$-equivariant localization, we find
$$
\left\langle \textup{weight} \, \cL \big|_{F_1} - \textup{weight} \,
\cL\big|_{F_2}, \sigma \right\rangle > 0 \,.
$$
\end{proof}

\subsection{Definition of $\Stab$}

\subsubsection{}

In addition to the set $\Attrc$ defined \eqref{Attrc}, we introduce
\begin{equation}
\Attrl= \bigcup_{F_j <  F_i} \sAttr(F_j) \times F_i \,.\label{Attrl}
\end{equation}
Note that the map $\jmath$ in \eqref{corrM} is proper on the
complement of $\Attrl$ and thus the section $[\Attr]$ in 
\eqref{[Attr]} is defined there. 

\subsubsection{}

\begin{Definition}
Let $\cS$ be a attractive line bundle for a given choice $\fC$ of
attracting directions. The elliptic stable envelope for $\cS$ is a
section
\begin{equation}
  \Stab : \cO_{\cE_\bT}  \to \cS \boxtimes \left(\cS_\bA\right)^\dd
  (\infty \bDel)
  \,. 
  \label{Stab}
  \end{equation} 
which is supported on
$$
\fAttr \subset \bX \times \bX^\bA
$$
and equals $[\Attr]$ on the complement of $\Attrl$. 
\end{Definition}

\subsubsection{}
In terms of algebraic geometry of the scheme $\Ell_\bT(\bX)$, stable
envelopes solve an interpolation problem: they take a given value
modulo the ideal $\cI_{<}$ which is the kernel of the restriction map
$$
0 \to \cI_{<}  \to \cO_{\Ell_{\bT}(\bX \times \bX^\bA)} \to
\cO_{\Ell_{\bT}(\bX \times \bX^\bA  \setminus \Attrl)} \,. 
$$
Existence and
uniqueness in interpolation of sections of a line bundle $\cL$ requires its degree
$\deg \cL$ to satisfy a certain balance: larger degree makes existence
easier and uniqueness harder, and vice versa. Our Definition
\ref{d_attr} provides the right balance. 

\subsubsection{}

The main goal of this paper is to give a direct proof of the following

\begin{Theorem}\label{t1} 
 Elliptic stable envelopes exist and are unique. In fact, they are
 unique among correspondences supported on the set
 $\Attrc$ defined in \eqref{Attrc}. 
\end{Theorem}

The proof of
uniqueness given in
Section 3.5 of \cite{ese} adapts to the setup of the present
paper. Existence of elliptic stable envelopes was shown in \cite{ese} for
Nakajima varieties using global techniques, namely abelianization
\cite{Shen}.
That line of reasoning has
the advantage of giving explicit formulas, see examples in \cite{ese}
and \cite{SmirHilb}. It also has its
limitations, which is why it is good to have an inductive (and, in that
sense, local) argument for existence. 

\subsubsection{}

We refer to \cite{ese} for examples of applications of elliptic stable
envelopes in enumerative geometry and geometric representation
theory. See also \cite{AO2} for their interpretation in mathematical
physics. 


\subsection{Proof of Theorem \ref{t1}} \label{s_proof1}

\subsubsection{}

Since long exact sequence of cohomology will be needed, we restore the
cohomological grading in $\cOd$ and $\cSd$. 

\subsubsection{} 

Since $\bX^\bA$ is disconnected, we may work with one component at a
time. We fix one such component $F_0$.

\subsubsection{}

Choose an arbitrary refinement of the partial order
\eqref{part_order} to a total order
and define 
$$
Y_i = \bigcup_{F_j  \le F_i} \sAttr(F_j)\,,  \qquad  Y_{<i} = \bigcup_{F_j <
  F_j} \sAttr(F_j)\,. 
$$
We denote by
$$
\bX_i = \bX \setminus Y_i \,, \quad \bX_{<i} = \bX \setminus Y_{<i}
$$
the complements of these sets. The sets $\bX_{<i}$ form an increasing
sequence of open set, eventually covering all of $\bX$.

\subsubsection{}

By definition, 
$$
\Stab\big|_{\bX_{<0} \times F_0} = [\Attr] \,. 
$$
The inductive construction of
stable envelopes refers to extending to all $\bX_{<i} \times
F_0$ by decreasing induction in $i$.

\subsubsection{} 

In one step of this induction, we abbreviate
$$
F = F_i
$$
and we redraw the diagram \eqref{corrM} as follows
\begin{equation}
  \label{corrM2}
  \xymatrix{
   \bX_{<i} \times F \ar[d]_{\forp_1} && \sAttr(F) \ar[dll]_\jb
   \ar[ll]_{\jmath} \ar[d]^\pi\\
   \bX_{<i}   && F \ar[ll]_\iota\,,  
    } 
  \end{equation}
  %
in which
$$
\jb = \forp_1 \circ \jmath
$$
is a proper embedding.

Let $N=N_{\bX/F}$ be the normal bundle to $F$ in $\bX$ and let
$N_{<0}$ denote the repelling part of this bundle. The normal bundle
to $\jb$ has the form $\pi^* {N_{<0}}$ and thus we have the map 
\begin{equation}
\jb\pf \pi^*: \Thetad(-N_{<0}) \to \cOd_{\Ell_\bT(\bX_{<j})}  \,. \label{pijmath}
\end{equation}

\subsubsection{}

\begin{Lemma}
The map of sheaves \eqref{pijmath} is injective. 
\end{Lemma}

\begin{proof}
Consider 
$$
\iota^* \jmath\pf \pi^*:  \Thetad(-N_{<0}) \to  \cOd_{\Ell_\bT(F)} \,. 
$$
This map is multiplication by the canonical section $\vth$ of
$\Thetad(N_{<0})$ that vanishes on the Chern roots of $N_{<0}$,
see \eqref{sTh}. 

The kernel of $\jmath\pf \pi^*$ has to be a subsheaf supported on 
$\{\vth=0\}$ 
but $\cOd_{\Ell_\bT(F)}$ has no such subsheaf. Indeed, since
$\bA$ acts trivially on $F$, $\cOd_{\Ell_\bT(F)}$ is pulled back
via the map $\phi$ in \eqref{pullA}.
On the other hand, $\{\vth=0\}$ contains no fibers of $\phi$ 
since all Chern roots of $N_{<0}$ are
nonzero when restricted to $\bA$, which means that the image of the
map 
$$
c(N_{<0}) : \cE_\bA \to S^{\rk N_{<0}} E
$$
is not contained in the divisor $D_\Theta$ in \eqref{DTh}. 
\end{proof}

\subsubsection{}
Consider the long exact sequence of the pair associate to the
embedding $\jb$ and note that
$$
\bX_i = \bX_{<i} \setminus \sAttr(F)
$$
by construction. 

Since the pushforward maps in this
long
exact sequence are injective, all connecting homomorphism are zero,
and we get the following

\begin{Proposition}
  We have a short exact sequence of sheaves
  \begin{equation}
    \label{shortexact}
    0 \to \Thetad(-N_{<0}) \xrightarrow{\,\, \jb \pf \pi^*\,}
    \cOd_{\Ell_\bT(\bX_{<i})} \to   \cOd_{\Ell_\bT(\bX_i)} \to 0  \,, 
  \end{equation}
  in which the surjection is pullback with respect to the open
  embedding
  $\bX_i \to \bX_{<i}$. 
\end{Proposition}

{} From this point on, we are interested only in the 0 cohomological
degree part of all sheaves, so we drop the degree grading. 

\subsubsection{}

We now consider the product $\bX_{<i} \times F_0$ and tensor
\eqref{shortexact} with $\cSd \boxtimes
\left(\cSd_\bA\right)^\dd$. We get
  \begin{equation}
    \label{shortexact2}
    0 \to \cS_{\bA, F_i} \boxtimes  \left(\cS_{\bA, F_0}\right)^\dd
    \to \cS \boxtimes \left(\cS_\bA\right)^\dd \Big|_{\bX_{<i} \times
      F_0} \to \cS \boxtimes \left(\cS_\bA\right)^\dd \Big|_{\bX_{i} \times
      F_0}
    \to 0  \,. 
  \end{equation}

\subsubsection{}

  Observe that the second term in
  $$
  \left(\cS_{\bA, F_0}\right)^\dd = \left(\cS_{\bA, F_0}\right)^{-1}
  \otimes \Theta(TF_0)
  $$
  is pulled back via $\phi$. Therefore
  $$
  \supp \Rd \psi_*  \left(\cS_{\bA,F_j} \boxtimes \left(
      \cS_{\bA,F_i} \right)^\dd  \right) \subset \bDel   \,. 
  $$
  This means that \eqref{shortexact2} induces an isomorphism
  $$
  H^0( \bX_{<i} \times
  F_0, \cS \boxtimes \left(\cS_\bA\right)^\dd  (\infty \bDel))
  \to
  H^0( \bX_{i} \times
  F_0, \cS \boxtimes \left(\cS_\bA\right)^\dd  (\infty \bDel)) \,, 
  $$
  which completes the induction step.

  \subsubsection{}
  This proves the existence and uniqueness for stable envelopes as
  correspondences supported on $\Attrc$. To show they are supported on
  $\fAttr$, one applies the uniqueness to the manifold $X \setminus
  \fAttr$.  This finishes the proof. 

  \section{Toric varieties and equivariant K-theory}

  \subsection{Cohomology vanishing}\label{s_Dellam} 

  \subsubsection{}
  Our next goal is to explain how the above line of reasoning may be
  adapted to equivariant K-theory. The first result we need
  is the toric analog of the vanishing \eqref{vanishA}.

  We note that
  the cohomology of the
  sheaves $\cO(\Delta_\lambda)$ that will appear below may be shown to
  vanish 
  in many different ways, including direct \v Cech complex computation
in the style of \cites{Dan,Ful},
or perhaps using the relation of $\cO(\Delta_\lambda)$ to multiplier ideal sheaves on toric
varieties as computed in \cite{Blickle}. Here we use orbifolds.

  \subsubsection{}
  Let $\cR$ be a ring with unit. 
  We denote
  \begin{equation}
  \fa = \cochar(\bA) \otimes_\Z \R \,, \quad \fa^* = \cha(\bA)
  \otimes_\Z \R \,. 
\label{deffa}
\end{equation}
  Let 
  $$
  \Delta = \conv(\{\eta_i\}) \subset \fa^*
  $$
  be a nondegenerate polytope with vertices in the weight lattice of $\bA$. Such
  polytope defines a toric variety $\bbA = \bbA_\Delta$ over $\cR$
  with an equivariant 
  line bundle $\cO(\Delta)$ as follows. 

  \subsubsection{}

  The toric charts $U_{\Delta'} \subset \bbA$ correspond to all
  nonempty faces $\Delta' \subset \Delta$ and
  \begin{equation}
    \Gamma(U_{\Delta'}, \cO(\Delta)) = \bigoplus_{\mu \in
      \Cone_{\Delta'} (\Delta) \cap \cha(\bA)} \cR \, a^\mu
    \,, \label{GammaU} 
  \end{equation}
  where
  $$
  \Cone_{\Delta'} (\Delta)  = \textup{tangent cone to $\Delta$ at
  $\Delta'$}  \, , 
$$
see Figure \ref{fpolytope}. Note, in particular, that $U_\Delta \cong \bA$.

\begin{figure}[!h]
  \centering
  \includegraphics[scale=0.64]{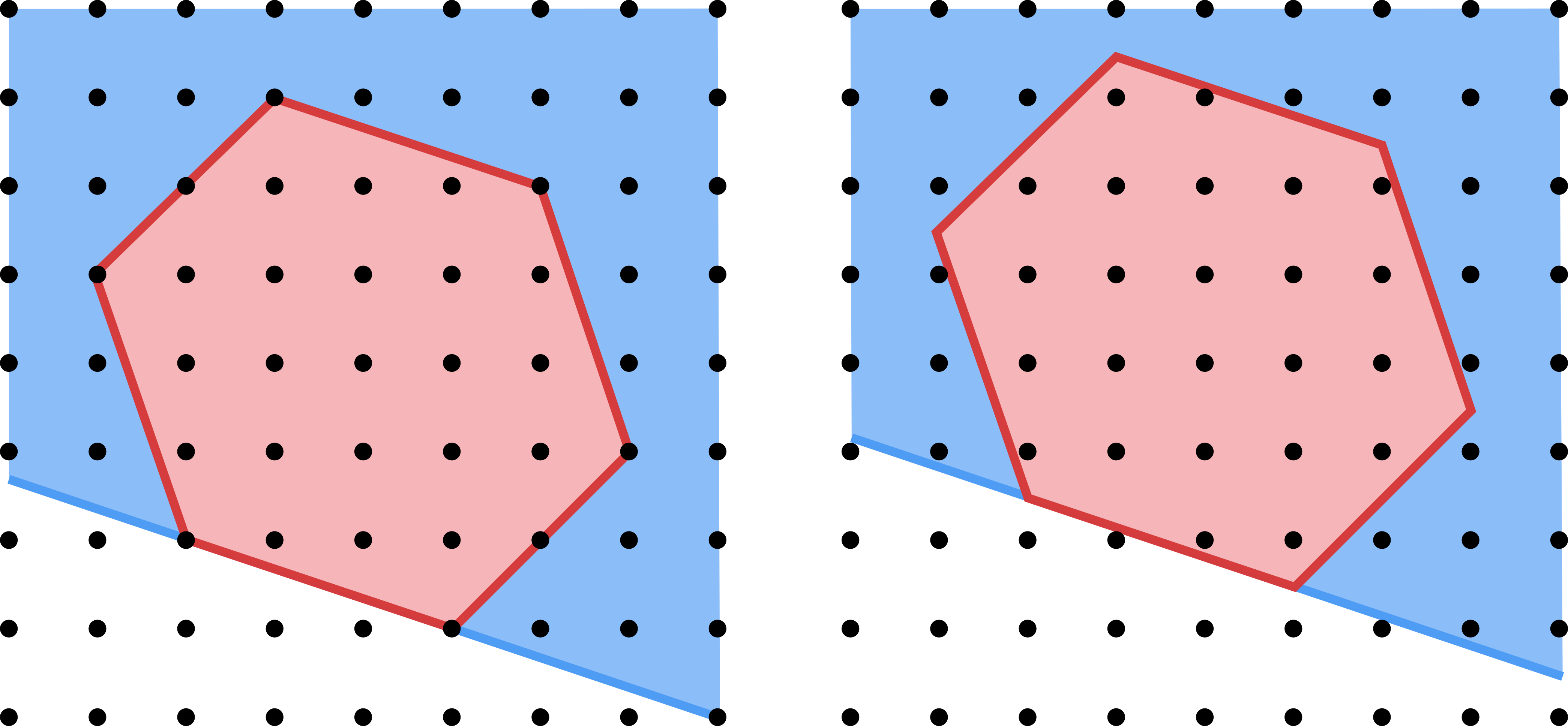}
  \caption{On the left, a lattice polytope with the cone at a
    1-dimensional face shaded blue. On the right, the same after a
    shift by a generic $\lambda$.}
\label{fpolytope}
\end{figure}

\subsubsection{}\label{s_O_Delta_lambda} 

Pick an arbitrary vector
$
\lambda \in  \fa^*
$
and let
$$
\Delta_\lambda = \Delta + \lambda
$$
be the translate of $\Delta$ by the vector $\lambda$. It is a
polytope, but no longer with integral vertices.

The same formula
\eqref{GammaU} defines a sheaf $\cO(\Delta_\lambda)$ on
$\bbA$. Note, however, that this sheaf may fail to be locally
free if $\bbA$ is singular.

\subsubsection{} 

We observe that
  \begin{equation}
    H^0(\bbA, \cO(\Delta_\lambda))= \cP_{\Delta_\lambda,\bA}
      \,, \label{H0D} 
  \end{equation}
  where for any $\Omega \subset \fa^*$ we denote by 
  $$
  \cP_{\Omega,\bA} = 
      \bigoplus_{\mu \in
        \Omega \cap \cha(\bA)} \cR \, a^\mu  \subset \cR[\bA] 
      $$
      the set of polynomials  with Newton polygon inside
      $\Omega$.

      \subsubsection{}
Clearly, there exist a locally finite periodic rational hyperplane
arrangement in $\fa^*$, such that $\cO(\Delta_\lambda)$ does not
change along the strata of this arrangement. Therefore, we may
assume, without loss of generality, that
$$
\lambda \in \fa_\Q^* = \cha(\bA)
  \otimes_\Z \Q \, .
  $$
  In this case, the toric variety $\bbA$ and the sheaf
  $\cO(\Delta_\lambda)$ may
  be defined in one step by
  \begin{equation}
  (\bbA, \cO(\Delta_\lambda)) = \Proj \bigoplus_{n\ge 0} \cP_{n
    \Delta_\lambda,\bA} \,.
\label{Proj1}
\end{equation}

  \subsubsection{}\label{s_orderN} 
  Let $N$ be the order of $\lambda$ in $\fa^*_\Q/\cha(\bA)$ and consider the
  torus $\bA_N$ with character lattice generated by $\lambda$ and
  $\cha(\bA)$. We have
  \begin{equation}
  1 \to \mu_N \to \bA_N \to \bA \to 1 \,, 
\label{bAN}
\end{equation}%
  by construction. Consider
  \begin{align}
  (\bbA_N, \cO(\Delta_\lambda)) &= \Proj \bigoplus_{n\ge 0} \cP_{n
                                  \Delta_\lambda,\bA_N}  \notag \\ 
                                  & \cong  (\bbA_N, \cO(\Delta)\,
                                    a^\lambda)\,, 
\label{Proj2}
\end{align}
where $a^\lambda$ is a character of $\bA_N$.  The following is
immediate

\begin{Lemma}
  We have
  \begin{equation}
    \label{eq:5}
      (\bbA, \cO(\Delta_\lambda)) =  (\bbA_N, \cO(\Delta) \, a^\lambda)
      \Big/ \!\!\!\! \Big/\mu_N \,. 
    \end{equation}
\end{Lemma}

  \subsubsection{}
  We abbreviate
  $$
  \cO(\Delta_\lambda - \Delta) = \cO(\Delta_\lambda) \otimes
  \cO(\Delta)^{-1} \,.
  $$
  The toric analog of \eqref{vanishA} that we will need is the
  following 
  \begin{Proposition}
    \begin{equation}
    H^i( \bbA,  \cO(\Delta_\lambda - \Delta)) = 
    \begin{cases}
      \cR \, a^\lambda\,, & i=0\,,   \lambda \in \cha(\bA) \,, \\
      0 \,, & \textup{otherwise} \,. 
    \end{cases}
\label{torvan}
\end{equation} %
\end{Proposition}

\begin{proof}
  On $\bbA_N$ we have
  $$
  \cO_{\bbA_N}(\Delta_\lambda-\Delta) = \cO_{\bbA_N} \, a^\lambda \,, 
  $$
  and therefore
  \begin{equation}
    H^i( \bbA_N,  \cO(\Delta_\lambda - \Delta)) = 
    \begin{cases}
      \cR \, a^\lambda\,, & i=0 \\
      0 \,, & i>0 \,. 
    \end{cases}
\label{torvan2}
\end{equation} %
Taking $\mu_N$-invariants concludes the proof. 
  \end{proof}




 \subsubsection{}\label{s_Pnondeg} 
 Let
 $$
 P(a) \in H^0(\bA,\cO(\Delta))
 $$
 be a polynomial with Newton polygon $\Delta$. We will call it
 \emph{nondegenerate} if the 
 coefficients of $P$ corresponding to the vertices of $\Delta$ are
 units in $\cR$.

 Consider the reduction modulo $P$ map 
 \begin{equation}
   \label{modP}
   H^0(\bA,\cO(\Delta_\lambda)) \to
   H^0(\bA,\cO(\Delta_\lambda)\big|_{P=0}) \,. 
 \end{equation}
 The vanishing \eqref{torvan} gives the following interpolation
 property for polynomials with Newton polygon inside
 $\Delta_\lambda$.

 \begin{Proposition}\label{p_int_toric} 
 Assume $P$ is nondegenerate. Then the map \eqref{modP} is
 surjective. It is also injective if $\lambda \notin \cha(\bA)$. If
 $\lambda$ is weight of $\bA$, then the kernel of \eqref{modP} equals $\cR
 a^\lambda P(a)$. 
\end{Proposition}

\begin{proof}
  The nondegeneracy of $P$ implies the sequence
  $$
  0 \to \cO(-\Delta) \xrightarrow{\, \, P \, \, } \cO \to \cO/P \to 0
  $$
  of sheaves on $\bbA$ is exact. Tensoring it with
  $\cO(\Delta_\lambda)$ and taking cohomology proves the proposition. 
\end{proof}

 \subsection{Stable envelopes in equivariant K-theory}
\label{sec:stable-envel-equiv}

\subsubsection{}
For simplicity, we assume that $X$ has a polarization $T^{1/2}$ with
respect to $\bA$. 
We refer to \cite{Opcmi} for a general discussion of stable envelopes in
equivariant K-theory in this context. 

In elliptic cohomology, we have the flexibility of twisting the
attractive line bundle $\cS$ by $\cU(L,z)$, where $L$ is 
a line bundle on $X$. We exploited this flexibility in Proposition
\ref{p_newS}. The corresponding parameter in equivariant K-theory is
a fractional line bundle
$$
L \in \Pic(X) \otimes_\Z \R
$$
known as \emph{slope}.

\subsubsection{}

By definition
$$
\Stab \in K_\bT(X \times X^\bA)
$$
is an extension of attracting manifold to a K-theory class
supported on $\Attr^f$ that satisfies
\begin{equation}
\deg_\bA \Stab\big|_{F_j \times F_i} \subset
\deg_\bA \Ld  \left(T^{1/2}\right)^\vee \big|_{F_j} +
\weight_\bA\left(L\big|_{F_j}\right) -
\weight_\bA\left(L\big|_{F_i}\right) \,.\label{degStab}
\end{equation}
Here
$$
\deg_\bA f(a) = \textup{Newton polytope} (f) \subset \fa^* 
$$
is the convex hull of nonzero coefficients of $f$.

\subsubsection{}

For a K-theory class supported on $\Attr^f$, the condition
\eqref{degStab} is equivalent to the, a priori, stronger condition
that
\begin{equation}
\deg_{\bA'} \Stab\big|_{F'\times F_i} \subset
\deg_{\bA'} \Ld  \left(T^{1/2}\right)^\vee \big|_{F'} +
\weight_{\bA'}\left(L\big|_{F'}\right) -
\weight_{\bA'}\left(L\big|_{F_i}\right) \,.\label{degStab2}
\end{equation}
for any subtorus $\bA' \subset \bA$ and any component $F'$ of 
restriction to $X^{\bA'}$. The polytopes in \eqref{degStab2} are the
projections of the corresponding polytopes in \eqref{degStab} to
$(\fa')^*$.

\subsubsection{}

Following the logic of the proof of Theorem \ref{t1}, we consider the
K-theory analog of \eqref{shortexact} 
  \begin{equation}
    \label{shortexact2}
    0 \to K_\bT(\Attr(F_i) \times F_0) \to 
    K_\bT(\bX_{<i} \times F_0) \to   K_\bT(\bX_i \times F_0) \to 0  \,. 
  \end{equation}
  In topological K-theory, we may replace $\bX_{<i}$ by
  \begin{align*}
    \bX'_{<i} &:= \textup{total space of  the normal bundle $\cN$ to
               $\Attr(F_i)$} \\
    \bX'_i &:= \textup{complement of the zero section of $\cN$} \,. 
  \end{align*}
  Then the sequence
  \eqref{shortexact2} becomes
  \begin{equation}
    \label{shortexact3}
    0 \to \cR[\bA] \xrightarrow{\,\, P \, \,} 
   \cR[\bA]  \to   \cR[\bA]/P\to 0  \,, 
  \end{equation}
  with
  $$
  \cR = K_{\bT/\bA}(F_i \times F_0)
  $$
  and
  \begin{equation}
  P = \sum (-1)^k \Lambda^k \cN^\vee \,. 
\label{PLambda}
\end{equation}

\subsubsection{} 

  \begin{Lemma}
  The polynomial \eqref{PLambda} is nondegenerate in the sense of
  Section \ref{s_Pnondeg}. 
\end{Lemma}

\begin{proof}
Let $\eta\in \fa^*$ be a vertex of the Newton polytope of $P$. Since
it is a vertex, there exists $\xi \in \fa$ such that $\eta$ is the
unique maximum of $\xi$ on the polytope.

Decompose $\cN$ according to the characters of $\bA$ 
\begin{equation}
\cN= \bigoplus a^\mu \, \cN_\mu \,, \label{Ndec}
\end{equation}
and note that
$$
\langle \mu, \xi \rangle \ne 0
$$
if $\xi$ is chosen generically.  Then 
$$
P = \pm \Lambda^{\textup{top}} \left( \bigoplus_{\langle \mu, \xi
    \rangle < 0} \cN^\vee_\mu \right) \, a^\eta + \dots
$$
where dots stand for terms of lower weight with respect to $\xi$. 
\end{proof}

\subsubsection{}
{} From \eqref{Ndec} we conclude
$$
P(a) = \prod_{\textup{indivisible $\nu$}} P_\nu(a^\nu) \,.
$$
This corresponds to the decomposition
\begin{align}
  \Spec K_\bT(\bX'_i \times F_0) &= \Spec \cR[\bA]/P  \notag \\
                                   & = \bigcup_{\textup{indivisible $\nu$}}
                                     \Spec \cR[\bA]/P_\nu \notag \\
  &= \bigcup_{\codim \bA'=1}  \Spec
  K_\bT((\bX'_i)^{\bA'} \times F_0)\label{decomXp}
  \end{align}
  where the bijection between indivisible characters $\nu$ and
  codimension one subtori $\bA' \subset \bA$ is given by
  $$
  1 \to \bA' \to \bA \xrightarrow{a^\nu} \Ct \to 1 \,.
  $$

\subsubsection{}
We now apply Proposition \ref{p_int_toric} with 
\begin{align*}
  \Delta & = \deg_\bA \Ld  \left(T^{1/2}\right)^\vee \big|_{F_i} \,, \\
  \lambda &= 
\weight_\bA\left(L\big|_{F_i}\right) - 
\weight_\bA\left(L\big|_{F_0}\right) \,. 
\end{align*}
By \eqref{degStab2}, \eqref{decomXp} and the inductive hypothesis, $\Stab\big|_{\bX'_i
  \times F_0}$ gives an element of $H^0(\bbA,
\cO(\Delta_\lambda)\big|_{P=0})$,
which can be lifted to an element of $H^0(\bbA,
\cO(\Delta_\lambda))$. This lift is unique if $\lambda$ is
not integral. This gives a proof of the following

\begin{Theorem}
Stable envelopes exist in equivariant topological
K-theory. Further, they are 
unique if
$$
\weight_\bA\left(L\big|_{F_j}\right) - 
\weight_\bA\left(L\big|_{F_i}\right) \notin \cha(\bA)
$$
for every pair $F_j < F_i$. 
\end{Theorem}

\subsubsection{}
The above argument may be also adapted to work in algebraic K-theory,
e.g. by first constructing stable envelopes for a generic attracting
cocharacter $\sigma: \Ct \to \bA$ and then checking the condition
\eqref{degStab2}
inductively on formal neighborhoods. However, a much stronger result
in $D^b \Coh X$ has been already established in \cite{DHLMO} and we refer
the reader there for details.

\subsection{Compactified K-theory}

The logic of Section \ref{sec:stable-envel-equiv} suggests the
following partial compactification of $\Spec K_\bT(\bX)$.

\subsubsection{}\label{s_intertia_fan}

Consider the fan in $\fa$ defined in Section \ref{s_fC} and call it
the \emph{inertia fan} of $\bX$. Recall that we have fixed a cone
$\fC$ of maximal dimension in this fan. This choice will \emph{not} play a role
in the current discussion. Let $\fC'$ be a cone of some dimension in
the inertia fan. It defines a subtorus $\bA' \subset \bA$, with
$$
\fa' =\Lie_\R \bA' = \Span(\fC')
$$
and a choice of attracting/repelling directions for the $\bA$-action
on $\bX^{\bA'}$.

One should think of each linear subspace $\fa'$ as being
decorated by the fixed locus $\bX^{\bA'}$ together with the weights of
its normal bundle. Those are the weights $\{w_i\}$ of Section
\ref{s_fC} that do not vanish on $\fa'$. The fan itself records the 
weights $\{w_i\}$ only up to proportionality and without multiplicity. 

\subsubsection{}

This fan gives a
toric compactification $\bAb$ of $\bA$, and thus a partial toric
compactification of $\bTb \supset \bT$.  Boundary strata correspond to
cones $\fC'$. As we approach a boundary
stratum
, we go to a particular infinity
$0_{\fC'}$ in
the torus $\bA'$. The coordinates in $\bT/\bA'$ are the coordinates
on the boundary stratum. 

For any splitting of 
\begin{equation}
  \label{TTT}
  1 \to \bA' \to \bT \to \bT/\bA' \to 1  \,, 
\end{equation}
we have
\begin{equation}
  \xymatrix{
      \Spec K_{\bT/\bA'}(\bX^{\bA'}) \ar[r] \ar[d] &   \Spec
      K_{\bT}(\bX) \ar[d] \\
     \textup{generic point}
      \ar[r] & \bA'  \,. 
  }\label{generic_point}
\end{equation}
With our hypotheses, the family \eqref{generic_point} extends over a neighborhood of  $0_{\fC'}$.
This is independent of the choice of the splitting 
and gives a partial compactification $\Kbar_\bT(\bX)$  of $\Spec
      K_{\bT}(\bX)$ over
$\bTb$  such
that 
\begin{equation}
  \xymatrix{
      \Spec
      K_{\bT}(\bX) \ar[d] \ar[r] & \Kbar_\bT(\bX) \ar[d] & 
      \Spec K_{\bT/\bA'}(\bX^{\bA'})  \ar[d] \ar[l] \\
     \bT \ar[r]  & \bTb &   \ar[l] \bT/\bA' \,, 
  } \label{KbarX} 
\end{equation}
where $\bT/\bA'$ is embedded as a boundary stratum. Note that
$K_\bT$ denotes a ring, while $\Kbar_\bT$ denotes a scheme. This makes
sense, because $\Kbar_\bT$ is not affine.

\subsubsection{}

Recall that a line bundle $\cL$ on a toric variety $\bTb$ is specified
by a function
\begin{equation}
\ord: \fa \to \R\label{ord}
\end{equation}
which is continuous and integral linear on the cones of the fan. The sections
$f$ of
$\cL$ are characterized by
$$
f = O(a^{\ord})\,, \quad a\to 0_{\fC'} \,, 
$$
for all cones $\fC'$. In what follows, we relax the integrality assumption
on the function \eqref{ord}. This allows the 
sheaves $\cO(\Delta_\lambda)$ considered in Section
\ref{s_O_Delta_lambda}.

To construct a line bundle on $\Kbar_\bT(\bX)$, we may use a different
functions \eqref{ord} on different components $F$ of the fixed locus,
as long as these glue over different strata. Given a polarization $T^{1/2}$
and a slope $L$, we define the line bundle $\cS(T^{1/2},L)$ by 
$$
\ord_{\cS(T^{1/2},L)} =  \textup{maximal weight in } L \otimes \left( \Ld \,
  T^{1/2}  \right)^\vee \,,
\quad  a\to 0_{\fC'}  \,. 
$$
With this definition, the degree bound \eqref{degStab2}  is equivalent
to being a global section of $\cS$. 
Note that the last term in \eqref{degStab2} is a constant from our
current perspective and may be absorbed into the choice of linearization of
$L$.


\subsubsection{}\label{s_replace_polarization}

Note we may replace the polarization $T^{1/2}$ by any other virtual
bundle $\cV$ on $\bX$, as long as the weights of $\cV$ define a fan which
refines the inertia fan. For instance, one can use the tangent bundle
to get a line bundle which is, basically, the square of $\cS$.

\subsubsection{}

Recall that, by definition
\begin{equation}
  \label{inducedpo}
  \textup{induced polarization of $\bX^{\bA'}$} =
  \left(T^{1/2}\big|_{\bX^{\bA'}}\right)^{\bA'} \,. 
\end{equation}
Similarly, $L\big|_{\bX^{\bA'}}$ gives the induced slope.

In the compactification \eqref{KbarX}, the same fixed locus $\bX^{\bA'}$ appears at several
infinities, and the natural polarization
of all these copies is different. 

\begin{Definition}
  Given $0_{\fC'} \in \overline{\bA'}$ and virtual vector bundle
  $\cV$, 
  we define 
\begin{equation}
  \label{limitpo}
  \cV_{\lim}= \cV\big|_{\bX^{\bA'}, \ge 0} - \cV^\vee
  \big|_{\bX^{\bA'}, < 0}\,, 
\end{equation}
where the subscripts refer to nonrepelling and repelling directions as
$a\to 0_{\fC'}$. 
\end{Definition}

\noindent
Note that we can write
\begin{equation}
  \label{limitpo2}
  \cV_{\lim}= \left(\cV\big|_{\bX^{\bA'}} \right)^{\bA'} 
    + \delta \cV - \delta \cV^\vee \,, \quad
 \delta \cV = \cV\big|_{\bX^{\bA'}, >0} \,. 
\end{equation}
In particular, the limit of a polarization is a polarization. It 
differs from the induced polarization precisely by the $\delta \cV - \delta \cV^\vee$  term
in \eqref{limitpo2}.

We have the following simple 

\begin{Lemma}\label{l_lim} 
 The boundary strata of $(\Kbar_\bT(\bX),\cS)$ have the form
 $(\Kbar_\bT(\bX^{\bA'}),\cS_{\lim})$, where $\cS_{\lim}$ is defined using
 the limit polarization and the induced slope. 
\end{Lemma}

\begin{proof}
  Follows from 
  $$
  1 - w^{-1} \sim
  \begin{cases}
    (1-w^{-1}) \big/ (1-w^{-1})  \,, & w\to \infty \,, \\
    (1-w^{-1}) \big/ (1-w) \,, & w\to 0 \,.
  \end{cases}
  $$
\end{proof}

\subsubsection{}

Recall that every change in polarization may be compensated by a
change in the slope. In particular, we have 
\begin{equation}
  \label{Ldplim}
  \deg _\bA \Ld \,\,  T^{1/2}_{\lim} = \deg_\bA \det N^{1/2}_{>0}  \otimes \Ld \, \, 
  T^{1/2}_\textup{ind} \,, 
\end{equation}
where  $N^{1/2}_{>0}$ is the attracting part of the
polarization restricted to $\bX^{\bA'}$. Therefore, 
$$
\cS_{\lim} = \cS\left(T^{1/2}_{\lim}, L_\textup{ind}\right)=
\cS\left(T^{1/2}_\textup{ind}, L_{\lim}\right) \,, 
$$
where
$$
L_{\lim} = L \otimes (\det
N^{1/2}_{>0})^{-1}  \,.
$$

\subsubsection{}

\noindent 
Recall from
formula (9.1.6) in \cite{Opcmi} that stable envelopes may be
normalized so that
\begin{equation}
  \label{norm_stab}
  \Stab\big|_{\diag \bX^\bA} = (-1)^{-\rk N^{1/2}_{>0}} \left(\frac{\det
      N_{<0}} {\det N^{1/2}} \right)^{1/2} \, \cO_\Attr \,,
\end{equation}
where $N^{1/2}$ denote the moving part of $ T^{1/2}\big|_{\bX^{\bA}}$ and
subscripts denote attracting and repelling directions for
the chamber $\fC$ as in Section \ref{s_fC}. The following is
straightforward: 

\begin{Lemma}\label{l_res_Attr}
The limit of \eqref{norm_stab} as $a\to 0_\fC'$ is the same expression for
the limit polarization. 
\end{Lemma}

\subsubsection{}

\begin{Proposition}\label{p_stable_compact}
Let $F$ be a component of $\bX^\bA$ and let the slope $L$ be generic
and linearized so that
it has zero $\bA$-weight on $F$. For any $\alpha \in K_{\bT/\bA}(F)$,
we have
\begin{equation}
\Stab_{T^{1/2},L} \, \alpha \in H^0\left(\Kbar_\bT(\bX), \cS(T^{1/2},L)\right)\label{StabH0}
\end{equation}
and its restriction to the boundary are the stable envelopes for
$\bX^{\bA'}$, extended by zero outside the component containing $F$. 
\end{Proposition}

\begin{proof}
Only the statement about restriction to the boundary requires
proof. Restriction to the boundary means taking the coefficient of 
$a^{\ord}$ as $a\to 0_{\fC'}$. It can only be nonzero if the function
$\ord$ is integral on $\fC'$. Therefore, by Lemma \ref{l_weight} and
genericity of $L$, 
the restriction to the boundary vanishes outside of the connected
component containing $F$.

On the connected component of $\Kbar_\bT(\bX^{\bA'})$ containing $F$,
the restriction of $\Stab_{T^{1/2},L} \, \alpha$ 
satisfies the degree bounds for stable envelopes.  By Lemma
\ref{l_res_Attr}, it correctly restricts to $F$. Thus it is the stable
envelope for the $\bA$-action on $\bX^{\bA'}$. 
\end{proof}

\section{Nodal degeneration of stable envelopes}
\label{sec:stable-envel-tate}

\subsection{$\Ell_\bT(\pt)$ near a node}

\subsubsection{} 

Equivariant cohomology theories over $\bB$ are constructed from $1$-dimensional
group schemes over $\bB$ and, in this section, we are interested in
the case when a nodal fiber $E_0$ appears in a family of elliptic
curves. In a neighborhood of the nodal fiber (formal or analytic),
there is a relation between $K_\bT(\bX)$ and $\Ell_\bT(\bX)$. Our goal is to explain this relation
and verify that it 
respects stable envelopes.

\subsubsection{}

Over a field, the identity component of the group smooth points of
$E_0$ is a  torus of rank 1, which may require a quadratic extension to split.
For simplicity, let us assume it is split. Moreover, 
as our local model, we will take $\bB = \Spec \cR[[q]]$, where $\cR$ is
a commutative ring with unit, and 
\begin{equation}
\xymatrix{
  E_0  \ar[r] \ar[d] & E \ar[d] & E_q = \Gm / q^\Z  \ar[l] \ar[d]  \\
  q=0 \ar[r] & \bB &  \Spec \cR\Fq  \ar[l] \,, 
  }\label{local model} 
\end{equation}
see Appendix \ref{s_Tate} for a reminder about Tate's construction
of $\Gm / q^\Z$.

Tate's construction only requires
the convergence of the standard $\vth$-series \eqref{theta_ser}.  Therefore, one could replace
$\cR[[q]]$ with any complete normed commutative ring with a
nonzerodivisor
 $q$ of norm $\|
q\| < 1$, for instance functions holomorphic in the unit disk of
$\C$. 

\subsubsection{}


For brevity, we denote $\cE=\Ell_\bT(\pt)$. 
We begin with general remarks about the structure of
$\cO_{\Ell_\bT(\bX)}$ as a sheaf over $\cE$. We observe
that it belongs to a certain abelian subcategory of
$\Coh \cE$, which is independent of $q$.

\subsubsection{}

The manifold $\bX$ may be represented by a cell complex  built from $\bT$-equivariant cells of the form
$\bT/\Gamma  \times D^n$, where $\Gamma \subset \bT$ is a
subgroup and $\bT$ acts trivially on the disc $D^n$. Note that only 
finitely many possible
stabilizers $\Gamma$ appear for a given $\bX$. We call them the
\emph{inertia subgroups} of $\bX$  and denote their set by
$\bGamma_\bX=\{\Gamma\}$. The sets
\begin{equation}
\{\log \Gamma \}_{\Gamma \in \bGamma_\bX} \subset \ft = \Lie_\R
\bT\label{lGX} \,, 
\end{equation}
where $\log \Gamma$ is defined in Section \ref{logexp}, form a
periodic stratification of $\ft$.

If $\bX$ is a smooth manifold, the stratification \eqref{lGX}
corresponds to a periodic hyperplane arrangement --- the $\bT$-weights
in the normal bundle to the fixed locus. In general (e.g.\ by an equivariant
embedding into a smooth manifold) we may enlarge the set $\bGamma_\bX$
so that it corresponds to a periodic hyperplane arrangement. We will
assume that this is the case and will refer to \eqref{lGX} as the  
\emph{periodic inertia fan} of $\bX$, compare with Sections \ref{s_fC}
and 
\ref{s_intertia_fan}.

\subsubsection{} 

Reflecting the structure of cell and of the attaching maps, the sheaf $\cO_{\Ell_\bT(X)}$ is
built from the  sheaves
$$
\cO_{\Ell_\bT(\bT/\Gamma)} = \iota_{\Gamma,*} \cO_{\Ell_\Gamma(\pt)}
\in \Coh \Ell_\bT(\pt)
$$
where
$$
\iota_\Gamma: (\Gamma, \pt) \to (\bT, \pt) 
$$
is an equivariant map. We will identify $\Ell_\Gamma(\pt)$ with its 
image under $\iota_\Gamma$.

\subsubsection{}

Let $\bGamma = \{ \Gamma \}$ be a lattice of subgroups of
$\bT$. The corresponding $\Ell_\Gamma(\pt)$ define a stratification of
$\cE$. Further, the subvarieties
\begin{equation}
\diag \Ell_\Gamma(\pt)  \subset \cE^k\label{strata_k}
\end{equation}
define a stratification of $\cE^k$, in which the dimension of
consecutive strata differ by $k$.

The following definition is inspired by Bezrukavnikov's definition of
perverse coherent sheaves, see \cite{ArinkBezr}. 

\begin{Definition}
Let $f: \cF_1 \to \cF_2$ be a morphism in $\Coh \cE$. We say that $f$
is locally constant with respect to $\bGamma$ if it is a restriction
to the diagonal in $\cE^k$ of a morphism of constant rank along the stratification
\eqref{strata_k} for $k=2$.  We denote the category of locally constant
sheaves and locally constant morphism by $\Coh_{\bGamma} \cE$. 
\end{Definition}

\begin{Proposition}
For any $\bT$-equivariant map
$
f: X \to Y
$
between finite cell complexes, the corresponding map
$$
\Ell(f): \cO_{\Ell_\bT(X)} \to \cO_{\Ell_\bT(Y)} 
$$
is locally constant with respect to the stratification generated by
$\bGamma_X$ an $\bGamma_Y$. 
\end{Proposition}

\begin{proof}
  Define
  $$
  X^{(2)} = X \times_{X/\bT} X\,, 
  $$
  where the product is over the quotient by $\bT$. This is built from
  the cells of the form $(\bT/\Gamma)^2 \times D^n$, with the
  attachment maps as before. 
  It has a natural $\bT^2$ action such that
  \begin{equation}
  (\bT, X) \xrightarrow{\quad \diag \quad} (\bT^2, X^{(2)})
\label{XtoX2}
\end{equation}
is an equivariant morphism. The stabilizers of points now have the
form $\diag \Gamma \subset \bT^2$. 
  
  The map $f$ as above gives rise to a  $\bT^2$-equivariant map
  $$
  f^{(2)}:   X^{(2)} \to  Y^{(2)} 
  $$
  and is obtained from it after the restriction to the diagonal
  \eqref{XtoX2}, which proves the proposition. 
\end{proof}

Intuitively, the proposition is clear, as the behavior of $f$ on
different strata is determined by the corresponding map between fixed
points. There are other ways to prove the proposition, for instance,
the restriction of $f$ to a any cyclic subgroup of $\cE$ should be
equivariant with respect to all automorphisms of that subgroup.

\subsubsection{}

The equations of $\Gamma\subset \bT$  and
$\Ell_\Gamma(\pt) \subset \Ell_\bT(\pt)$ are encoded by the
kernel in the following exact
sequence
$$
0 \to \Gamma^\perp \to \chr(\bT) \to \chr(\Gamma) \to 0 \,. 
$$
In particular, the divisibility of the generators of $\Gamma^\perp$
determines the numbers $\{m_i\}$ in
\begin{equation}
\Ell_\Gamma(\pt) \cong E^{\dim \Gamma} \times \prod
E[m_i]\label{EllGa} \,, 
\end{equation}
where  $E[m]$ denotes the scheme of points
of order $m$ on $E$.

\subsubsection{}

 The category $\Coh_{\bGamma} \cE$ is determined by the structure of the
 lattice $\bGamma$ and the structure of subschemes $E[m] \subset E$ of
 points of order $m$. Since $E$ is a Tate elliptic
curve, we have\footnote{In particular, a Tate
  elliptic curve is never supersingular, which is the place where we
  get extra morphisms and, hence, more information can be captured by
  elliptic cohomology.}
$$
E[m] \cong \mu_m \times \Z/m \Z \,. 
$$
Here $\mu_m \subset \Gm$ is the scheme of $m$th roots of unity and
$\Z/m \Z$ is a constant group scheme which records the valuation a
point of order $m$ as in \eqref{val}. As a result, the category $\Coh_{\bGamma} \cE$ is
independent of $q$.


\subsubsection{}

Our next goal is to construct a degeneration of $\Ell_\bT(\pt)$ in
which the inertia subgroups $\Ell_\Gamma(\pt)$ 
are transverse to the special fiber and  have
the exact same structure (in particular, intersect each other in the exact
same way) in the special and the generic fibers. This will extend the
constant family of subcategories $\Coh_{\bGamma}(\cE)$  to the central
fiber. 

While this cannot be achieved for all subgroups of $\bT$, for any
finite set $\bGamma$ such degeneration may be
constructed after a base change of the form $q=(q^{1/M})^M$ and a birational
transformation. An elementary example
is plotted in Figure \ref{fLeg}.

\begin{figure}[!h]
  \centering
  \includegraphics[scale=0.33]{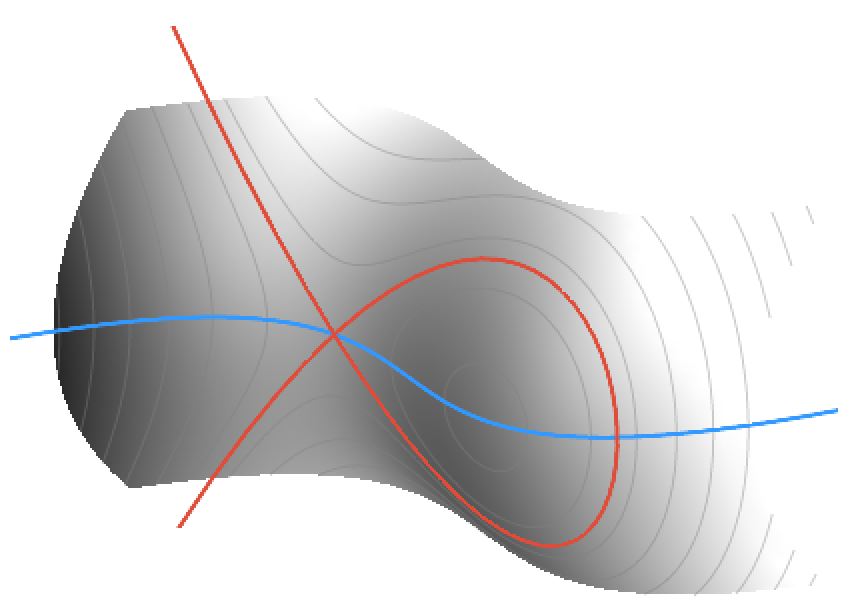}
  \hspace{0.5cm}\raisebox{-0.49cm}{\includegraphics[scale=0.33]{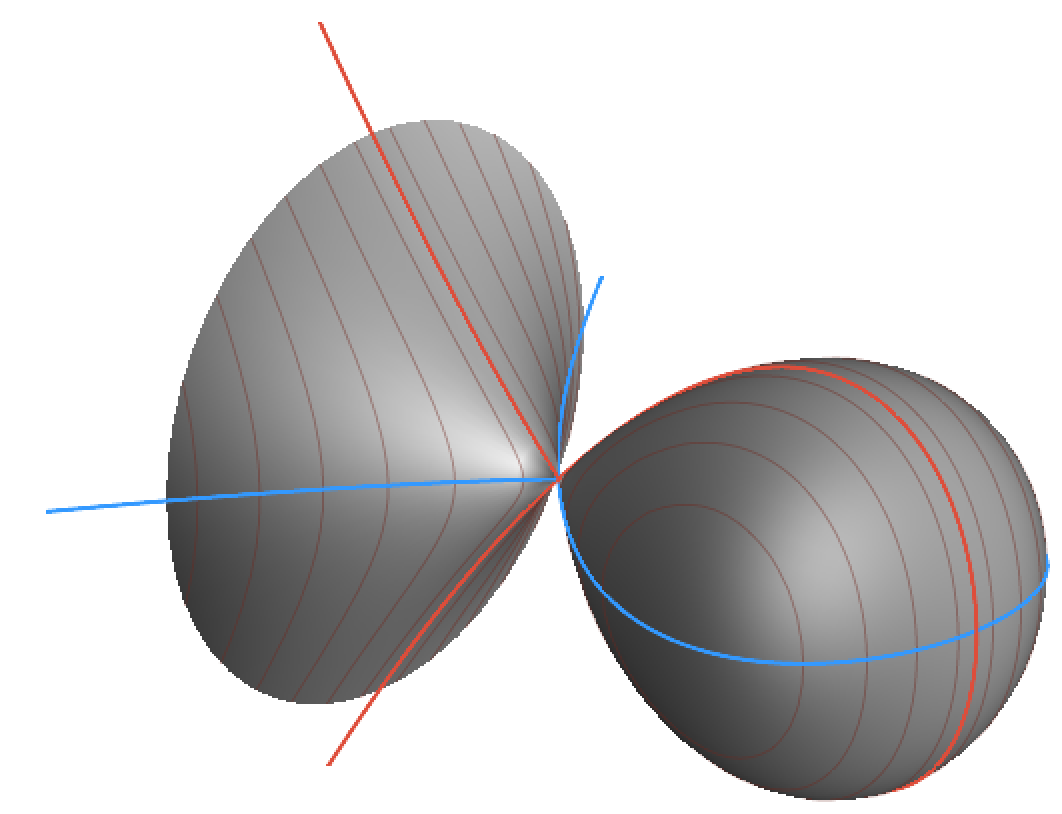}}
 \caption{In the Legendre family of elliptic curves (left), the
   subscheme $E[2]$ (plotted in blue) is tangent to the nodal fiber (red). After a base
   change $q=(q^{1/2})^2$ we get a surface which looks like the
   surface of revolution of the special fiber (right). Blowing up the
   singular point, we get $E[2]$ to intersects the nodal fiber
   transversally in 4 points, counting the identity point at
   infinity. The new nodal fiber is a union of two $\bP^1$ glued at
   two points.}
\label{fLeg}
\end{figure}

\subsubsection{}

In general, a suitable degeneration of $\Ell_\bT(\pt)$ can be
constructed in one step as follows, see Appendix \ref{s_cQ}.
The vector space
$$
\ft = \cochar(\bT) \otimes_\Z \R 
$$
will for now play the role of $\fa$ in Appendix \ref{s_cQ}. 

The construction of $\cE_\cQ$ in Appendix \ref{s_cQ} involves the following choices:
\begin{itemize}
\item a hyperplane arrangement $\mu_i(x) \in \Z$ in $\ft$, which
  refines the periodic fan of $\bX$, 
\item a line bundle $\cQ$ on the nodal fiber, which
  gives the pieces of the dual periodic tessellation of $\ft^*$ their particular shape and position.
\end{itemize}

 We call this data the \emph{periodic fan} and \emph{periodic
   tessellation}, respectively. The former
 describes which weights $\mu_i$ appear in
  \begin{equation}
    \cQ = \sum m_i \, \Qq(\mu_i) + \lambda\,, \quad \lambda \in
    \ft^*_\Q \,, 
\label{excQ}
\end{equation}
while the latter determines the
  multiplicities $m_i$ and the
  fractional linear shift $\lambda$.

  \subsubsection{}

Degenerations of abelian varieties with an ample line bundle is a 
well developed theory due to V.~Alexeev and his predecessors, see
\cite{Alex}. For our bookkeeping purposes,
we allow the fractional shift $\lambda$ and thus the orbifold line bundles of the kind discussed in Section
\ref{s_orderN} and \ref{s_fraction}. 

\subsubsection{} 

The strata $\eta$ of the period fan and the pieces $\eta^\vee$ of the
periodic tessellation are in a natural bijection. They also correspond
to the toric strata $\left\{ O_{\eta} \right\}$ of the nodal fiber of
$\cE_\cQ$.  Those have
the form
$$
O_{\eta} = \bT/\exp(d\eta)\,, 
$$
where $d\eta \subset \ft$ is the tangent space to $\eta$ and
$\exp(d\eta)$ is the corresponding subtorus as in \eqref{logGamma}.
See Appendix \ref{s_Gammaeta}.

\subsubsection{}
As everywhere else in this paper, our focus will be on a special
subtorus $\bA \subset \bT$. In particular, we will focus on the
restriction to $\fa \subset \ft$ of the hyperplane arrangement
corresponding to $\bGamma$. In terms of $q\to 0$ limits of
$\vartheta$-functions, this means we assume $\bnu(t_i)=0$,
that is,
$$
\log t_i = o(\log q)\,,  \quad q\to 0 \,, 
$$
for all
characters $t_i$ of $\bT/\bA$.

\subsection{Nodal K-theory}

\subsubsection{}


\begin{Definition}
   We define
the nodal K-theory $\Kn_\bT(\pt)_\cQ$ as the central fiber of
$\cE_\cQ$.
\end{Definition}

\subsubsection{}

The stratifications $\{O_{\eta}\}$ and $\{\Ell_\Gamma(\pt)\}$ are
transverse in the sense that the latter intersects
the nodal fiber transversely and 
\begin{equation}
O_{ \Gamma, \eta} \overset{\textup{\tiny def}}= \Ell_\Gamma(\pt)
\cap O_{\eta} \cong 
\begin{cases}
  \Gamma/\exp(d\eta)\,, &  \eta \subset \log \Gamma \,, \\
  \varnothing \,, & \textup{otherwise} \,. 
\end{cases}\label{transv}
\end{equation}

\subsubsection{}\label{s_defKbar0}

By the transversality \eqref{transv}, the construction of
$\Ell_\bT(\bX)$ canonically extends to the zero fiber, and we get
a scheme
$$
\Kn_\bT(\bX)_\cQ  \to \Kn_\bT(\pt)_\cQ \,.
$$
The neighborhood of any stratum $\eta$ in the periodic fan determines a fan in
$\ft/d\eta$. This fan refines the inertia fan of
$$
\bX^\eta \overset{\tiny\textup{def}} = \bX^{\exp(\eta)} \,, 
$$
and hence we can use it to construct
$\Kbar_{\bT/\exp(d\eta)}(\bX^{\eta})$. The following is clear. 

\begin{Proposition} For all $\eta$, we have the following pullback diagram
  \begin{equation}
    \label{unionKbar}
    \xymatrix{ 
       \Kbar_{\bT/\exp(d\eta)}(\bX^{\eta})  \ar[r] \ar[d] &\Kn_\bT(\bX)_\cQ \ar[d] \\
      \overline{O}_\eta\ar[r] & \Kn_\bT(\pt)_\cQ  \,. 
      } 
    \end{equation}
  \end{Proposition}

\subsubsection{}

There is a finer stratification by subschemes of the form
\begin{equation}
K_{\Gamma, \eta,  i} \to O_{\Gamma, \eta} \,, \quad
K_{\Gamma, \eta,  i}
= \Spec K_{\Gamma/\exp(d \eta)}(F_i)
\,, \label{strataKbar}
\end{equation}
where $F_i$ are the connected components of $\bX^{\eta}$.

\subsubsection{Example} 

Suppose $\bT\cong \Ct$, $\cR$ is a field of $\chr\ne 2$,  and
$$
\bGamma = \{1, \mu_2, \bT\}\,, 
$$
where $\mu_2=\{\pm 1\}$ is the group of square roots of
unity. We have 
$$
\log \mu_2 = \{0, \tfrac12\}  + \Z \,, 
$$
and points $\eta \in \log \mu_2$ correspond to two kinds of 1-dimensional strata $\eta^\vee \subset
\ft^*$.  The intervals between points in $\log \mu_2$
correspond to $0$-dimensional strata in $\ft^*$. The combinatorics of
the different strata is illustrated in Figure \ref{ffloor}.

\begin{figure}[!h]
  \centering
  \includegraphics[scale=0.75]{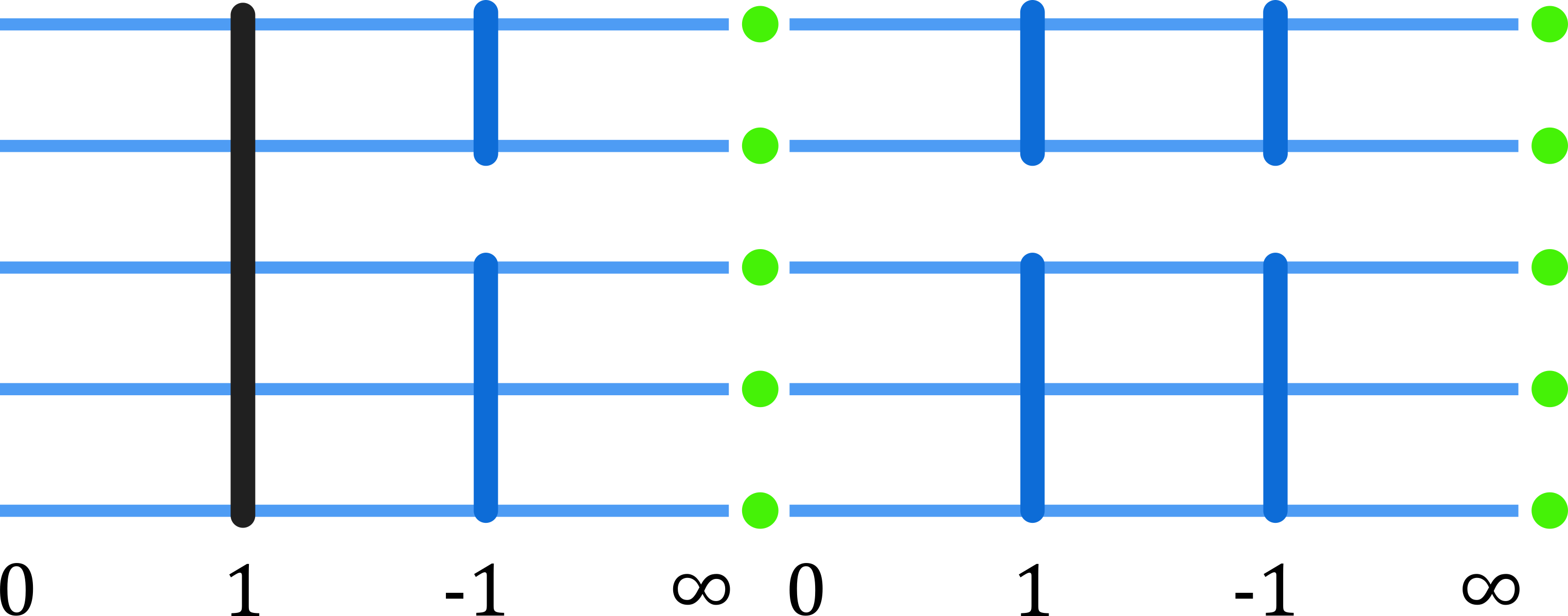}
  \caption{The combinatorics of the strata of $\Kn_{\bT}(\bX)$, assuming
    $\bX$ is connected, $\bX^{\mu_2}$ has two components, which
   break into 2 and 3 components, respectively, in $\bX^\bT$. Strata of
   the form $K_{\bT, \eta,  i}$ are indicated by connected
   components; there are three 1-dimensional ones and ten
   0-dimensional ones among them. The generic fiber over each
   $O_{\eta}$ is $\Spec K(\bX^\bT)$. We have
   $K_{1,0} \cong \Spec K(\bX)$, and the three remaining fibers are
   isomorphic to $\Spec K(\bX^{\mu_2})$.}
\label{ffloor}
\end{figure}

Strata over $O_{0}$ and $O_{\frac12}$ have the form $\Spec K_\bT(\bX)$ and
$\Spec K_\bT(\bX^{\mu_2})$, respectively, where the latter is a double cover of
$\Spec K_{\bT/\mu_2}(\bX^{\mu_2})$. In particular, the fiber $\Spec K(\bX^{\mu_2})$ appears
for a total of three times, corresponding to the nontrivial points of
order 2 of $\Ell_\bT(\pt)$. 

Also note that, as the picture is trying to suggest, the preimage of the
generic point of $ O_{\Gamma, \eta}$ need not be dense in
$K_{\Gamma, \eta,  i}$. For instance, the $\bT$ action on $\bX$ may be
free, in which case $\Kn_{\bT}(\bX)= K_{1,0}$. 

\subsubsection{}

A choice of a component of $X^\bT$ gives a consistent choice of
a component of $X^\eta$ for every $\eta$. The corresponding reduced
subschemes of $\Kn_\bT(X)_\cQ$ assemble into a copy of
$\Kn_\bT(\pt)_\cQ$. 

In a superficial parallel with the theory of buildings, these reduced
subscheme perhaps
may be called the \emph{floors} of $\Kn_\bT(X)$.  The subsets
\eqref{unionKbar} of maximal dimension,
which extend over several floors and share walls with their neighbors, may
probably be called the \emph{units} of $\Kn_\bT(X)$.  If
nodal K-theory proves to be a useful notion, people may 
find it
convenient to refer to different parts of $\Kn_\bT(X)$ as slabs,
beams, stairwells etc. A floor plan of a real $\Kn_\bT(X)$ may be
seen in Figure \ref{ftriangle}.

\subsubsection{}

Conversely, given
$$
\cO_{\Kn_{\bT}(X)_\cQ}\in \Coh_{\bGamma} \Kn_\bT(\pt)_\cQ
$$ 
it extends canonically in a constant way discussed above to $\cO_{\Ell_\bT(X)}\in \Coh_{\bGamma} \cE$
for a Tate elliptic curve. This is, effectively,
Grojanowski's original construction of the equivariant
elliptic cohomology over $\C$ in a Tate curve context.

\subsection{Theta bundles and stable envelopes}

\subsubsection{}

As defined in Section \ref{s_defKbar0}, the scheme $\Kn_{\bT}(\bX)_\cQ$
has a stratification and a line bundle inherited from
$\Kn_{\bT}(\pt)_\cQ$.  This construction may be generalized as
in the construction of compactified K-theory, see in particular Section
\ref{s_replace_polarization}. 

Let $\cV\in K_\bT(X)$ be a virtual vector bundle on $\bX$ such that all normal weights to
$\bX^\bT$ appear as weights of $\cV\big|_{\bX^\bT}$.  For instance, one 
can take $\cV=T\bX$. If we only compactify directions in $\bA$, we may
take
$$
\cV = T^{1/2} \,.
$$
We consider the periodic fan in $\ft$ defined by the weights
of $\cV\big|_{\bX^\bT}$. 


\subsubsection{}

For every stratum $\eta$ in the periodic fan, set
$$
\cV^\eta = \left(\cV\big|_{\bX^\eta} \right)^\eta
$$
Let $\eta\to \eta'$ indicate that $\eta$ is in the closure of $\eta'$.
Note the reversal of the arrow, which reflects the fact that 
$O_{\eta'}$ is in the closure of $O_{\eta}$. As in
Lemma \ref{l_lim}, we have
$$
\lim_{\eta\to \eta'} \Ld \,\,  \cV^{\eta} = \pm \det
\delta_{\eta,\eta'} \cV \otimes \Ld \,\,  \cV^{\eta'} \,, \quad
\delta_{\eta,\eta'} \cV = \cV^\eta\big|_{\bX^{\eta'},>0} \,. 
$$

\begin{Lemma}
There is a unique, up to multiple, assignment of a line bundle
$\bla_\eta \in \Pic_\bT(\bX)$ to each stratum $\eta$ such that
\begin{equation}
\det
\delta_{\eta,\eta'}  = \bla_{\eta} \otimes \bla_{\eta'}^{-1} \,. \label{bnubnu}
\end{equation}
\end{Lemma}

\begin{proof}
We need to show that $\delta_{\eta,\eta'}$ is a trivial 1-cocycle on
the adjacency graph of the strata $\eta$. Loops in this graph are
generated by triangles of the form $\eta\to \eta' \to \eta''$. Since
$$
\lim_{\eta \to \eta''} = \lim_{\eta' \to \eta''}  \circ \lim_{\eta \to
  \eta'}
$$
we have
$$
\det
\delta_{\eta,\eta''} = \det
\delta_{\eta,\eta'} \otimes \det
\delta_{\eta',\eta''}
$$
thus showing the triviality of the cocycle. 
\end{proof}

\noindent 
We will normalize the choice of $\bla$ so that $\bla_0$ is trivial. 

\subsubsection{}

\begin{Definition}
  We define the nodal K-theory of $\bX$ as the union
  \begin{equation}
  \Kn_\bT(\bT)_{\cV}  = \bigcup_\eta \Kbar(\bX^\eta)_{\cV^\eta,
    \bla_\eta} \,.
\label{KnX}
\end{equation}
  Given a a fractional line bundle $L\in \Pic_\bT(X) \otimes \Q$, we
  denote by $\Thn(\cV,L)$ the orbifold line bundle on \eqref{KnX}
  obtained by gluing $\cS(\cV^\eta, L \otimes \bla_\eta)$. 
\end{Definition}

\noindent
Note that for every pair $\eta \to \eta'$ there is an inclusion among
the corresponding terms in \eqref{KnX}. Thus \eqref{KnX} is
effectively a union over $0$-dimensional strata $\eta$.

\subsubsection{Example}

Take $\bX=\bP^2$ and $\cV=T \bX$, with the action of
$$
\bT = \diag (a_1, a_2, a_3) \subset PGL(3) \,. 
$$
Its weights at fixed points are cyclic permutations of
$\{a_2/a_1,a_3/a_1\}$, and the corresponding tiling of each floor is
by parallelograms with these sides (often called lozenges), see Figure
\ref{ftriangle}.

Since $\{a_2/a_1,a_3/a_1\}$ is a basis of $\chr(\bT)$, each lozenge corresponds to a copy of $\bP^1 \times
\bP^1$ in $\Kn_{\bT}(\bX)$ with a line bundle that is
isomorphic to $\cO_{\bP^1}(1) \boxtimes \cO_{\bP^1}(1)$. 

The inertia lattice $\bGamma$ has the form 
$$
\bGamma = \left\{ \raisebox{1.5cm}{\xymatrix{
 & \bT \\
  \{a_1=a_2\}  \ar[ur] &  \{a_1=a_3\} \ar[u]  &  \{a_2=a_3\} \ar[ul]\\
 & \{1\} \ar[ul] \ar[u] \ar[ur] 
} }\right\} \,. 
$$
The stratum corresponding to the trivial subgroup $\{1\}$ is at the
center of each lozenge. All three floors are glued together there like
the three coordinate 2-planes in 3-space. Two
floors are glued together where $a_i=a_j$. See Figure
\ref{ftriangle}. 

\begin{figure}[!h]
  \centering
  \includegraphics[scale=0.16]{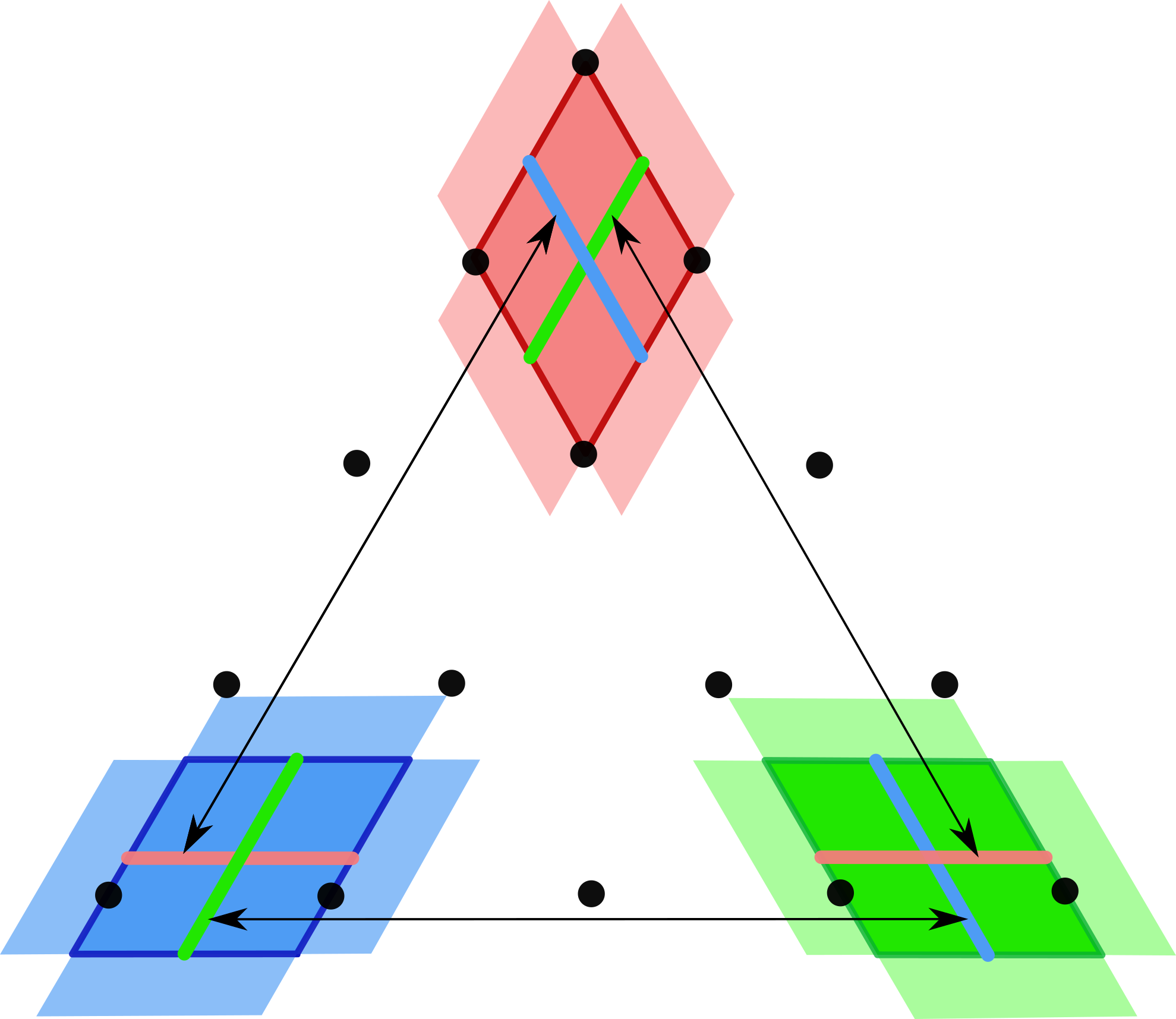}
 \caption{The floors of $\Kn_{\bT}(\bP^2)_{T\bP^2,L}$ plotted with
   respect to 
   the lattice $\chr \bT\subset \ft^*$. Each floors is tiled by lozenges with sides
   that are cyclic permutations of  $\{a_2/a_1,a_3/a_1\}$. The tilings
   are shifted with respect to each other by a weights of a fractional
   multiple $L$ of $\cO_{\bX}(1)$, visible as a large triangle (i.e.\ the
   toric polytope of $\bX$ itself) of
   fractional size. Arrows indicate gluing. Note that since this
   picture is in the dual of $\Lie \bT$, the gluing is not by
   embedding, but rather by projection along $\Gamma^\perp$. 
   }
\label{ftriangle}
\end{figure}

\subsubsection{}

\begin{Proposition}
The scheme $\Kn_{\bT}(\bX)_{\cV}$ with the line bundle $\Thn(\cV,L)$ deforms to $\Ell_\bT(\bX)$ with a
line bundle which is a K\"ahler shift of $\Theta(\cV)$. 
\end{Proposition}

\begin{proof}
The cell decomposition of $\bX$ exhibits $\Kn_{\bT}(\bX)_{\cV}$ as
built out of pieces $\Kn_{\Gamma}(Y)_{\cV}$, where $\Gamma$ acts
trivially on $Y$. Thus each of the pieces 
is a constant thickening of a degeneration of 
$\Ell_{\Gamma}(\pt)$, polarized by a K\"ahler shift of $\Theta(\cV)$. Since the deformation is canonical, it is
respected by the attachment maps. 
\end{proof}



\subsubsection{}
Recall from the discussion of compactified K-theory that it is
possible to compactify only directions in a subtorus $\bA \subset
\bT$. We can do the same in \eqref{KnX}, in which case the periodic
fan lies in $\fa$ and we can take $\cV= T^{1/2}$. 

\subsubsection{}
  \begin{Theorem}
    Stable envelopes for $\bA$-action on $\bX^\eta$ glue to a section of
    $\Thn(T^{1/2},L)$ and deform to elliptic stable envelopes for
    $\bX$. 
  \end{Theorem}

  \begin{proof}
    The first claim follows from Proposition
    \ref{p_stable_compact}. Inductive construction of stable envelopes
    is based on cohomology vanishing, which is semicontinuous. Thus they
    deform to sections of  a K\"ahler shift of $\Theta(T^{1/2})$. By
    uniqueness, these sections are elliptic stable envelopes for
    $\bX$. 
  \end{proof}

  This results differs by a change of perspective only from an earlier
  result of Kononov and Smirnov, proven in \cite{KonSmi} with somewhat different
  hypotheses and using different means.

  \begin{Theorem}{\cite{KonSmi}} The nodal limit of elliptic stable envelopes for
    $\bX$ gives elliptic stable envelopes for $\bX^\Gamma$, $\Gamma
    \subset \bA$. 
  \end{Theorem}

  \appendix

  \section{Constructions in elliptic cohomology}\label{s_constr}

 \subsection{Chern and Thom classes} 

\subsubsection{Chern classes}

A $\bT$-equivariant  complex vector bundle $V$ of rank $r$ over $\bX$ defines a map
\begin{equation}
c: \Ell_\bT(\bX) \to \Ell_{GL(r)}(\pt) = S^r E \,, \label{defc}
\end{equation}
see Section (1.8) in \cite{GKV} and Section 5 in \cite{Ganter}. 
The coordinates in the target of \eqref{defc} are 
symmetric functions on $E^r$ --- symmetric functions in elliptic 
Chern roots.

\subsubsection{Thom classes}

The Thom class of $V$ is, by definition, 
$$
\Theta(V) = c^* \cO(D_\Theta) 
$$
where $c$ is the map \eqref{defc} and 
the divisor 
\begin{equation}
D_\Theta = \{0\} + S^{r-1} E \subset  S^{r} E\label{DTh}
\end{equation}
is 
formed by those $r$-tuples that contain $0$.

\subsubsection{} 

Note that since
$D_\Theta$ is effective,  we have a canonical section 
%
\begin{equation}
\vth_V: \cO_{\Ell_\bT \bX} \to \Theta(V)\label{sTh} \,. 
\end{equation}
This section may be interpreted as the Euler class of the vector bundle $V$ --- the
elliptic cohomology class assigned to the locus cut out by a regular section of
$V$. 

\subsubsection{}

Since 
$$
\Theta(V_1 \oplus V_2)  = \Theta(V_1) \otimes \Theta(V_2) \,, 
$$
we have a group homomorphism
\begin{equation}
\Theta: K_\bT(\bX) \to \Pic \left(\Ell_\bT(\bX)\right)  \,. 
\label{Thom_map}
\end{equation}
Also observe that
\begin{equation}
\Theta(V^\vee)= \Theta(V) \label{ThVdual}
\end{equation}
because the divisor \eqref{DTh} is centrally symmetric. The canonical
section \eqref{sTh} changes sign by $(-1)^{\rk V}$. 

\subsubsection{}

The rank $\rk V$ of a K-theory class on $\bX$ is a locally constant
function $\bX \to \Z$, that is, an element of $H^0(\bX,\Z)$. 
We denote
$$
\fmo  = \Ker (K_\bG(\bX) \xrightarrow{\, \, \rk \, \,} H^0(\bX,\Z)) \,. 
$$
This ideal is the first step in the filtration of $K_\bT(\bX)$ by the
codimension of support.

\begin{Lemma}\label{l_cube} 
  The map $V \mapsto \Theta(V)$ factors through the quotient
  $K_\bG(\bX)/\fmo^3$. 
\end{Lemma}

\noindent See Appendix \ref{s_A1} for the proof.

\subsection{Fixed loci}

\subsubsection{}
We have two natural maps of the form  \eqref{map_pairs} associated to the
fixed locus $\bX^\bA$ namely
\begin{alignat}{2} 
  \label{incl}
  (\bT, \bX^\bA) & \xrightarrow{\quad\,\, 1_\bT \times \iota\quad\,\,\, \,} &&\, (\bT, \bX) \,, \\
                 (\bT, \bX^\bA) & \xrightarrow{\quad\phi\times 1_{\bX^\bA}\quad} &&\,(\bT/\bA, \bX^\bA) \,.
\end{alignat}
where $\iota: \bX^\bA \to \bX$ is the inclusion and $\phi: \bT\to \bT/\bA$ is the
quotient.

\subsubsection{}

The diagram 
\begin{equation}
  \label{pullA}
  \xymatrix{
    \Ell_\bT(\bX^\bA) \ar[rr]^{\phi} \ar[d]_p && \Ell_{\bT/\bA}(\bX^\bA) \ar[d]^p \\
    \cE_\bT \ar[rr]^{\phi} && \cE_{\bT/\bA}
    }\,, 
  \end{equation}
  is a pullback diagram. Its  horizontal maps are quotients by a 
  free action of $\cE_A$.

  \subsubsection{}

Given a line bundle $\cL$ on $\Ell_\bT(\bX)$, we denote by $\deg_\bA$
the degree of its restriction to the fibers of $\phi$ in
\eqref{pullA}, see Appendix \ref{A_degree}. This degree is an element of $S^2
\cha(\bA)$ which depends on the component of the fixed locus. In other
words:
\begin{equation}
  \label{degbA}
  \deg_\bA : \Pic(\Ell_\bT(\bX)) \to H^0(X^\bA, S^2
\cha(\bA)) \,. 
\end{equation}
For instance, 
\begin{equation}
\deg_\bA \Theta(V) =( c_1^2 - 2 c_2) \left(V|_{\pt \in X^\bA}\right)
\,. \label{degThV}
\end{equation}
where we take equivariant Chern classes of the fiber of $V$ at a point of
$X^\bA$. Formula \eqref{degThV} follows from \eqref{degThV2}. 

\subsection{Push-forwards}

\subsubsection{}\label{s_push}

Recall the functor
\begin{equation}
  \label{f_*}
f_*: \Coh \Ell_{\bG_1}(\bX_1) \to  \Coh \Ell_{\bG_2}(\bX_2)
\end{equation}
induced by a map of the form \eqref{map_pairs}.



Pushforwards in equivariant elliptic cohomology are defined for proper complex
oriented equivariant maps, and are homomorphisms
\begin{equation}
f\pf \in \Hom_{\Coh \Ell_{\bG_2}(\bX_2)}
(f_* \Theta_{\bX_1}(-N_f),\cO_{\Ell_{\bG_2}(\bX_2)}) \label{fpf}
\,, 
\end{equation}
where $N_f \in K_\bT(\bX_1)$ is the normal bundle to $f$.
\subsubsection{} 

Note the 
difference
$$
f\pf \ne f_*  \,. 
$$
For starters, $f_*$ is a functor, while $f\pf$ is a map in
the target category of $f_*$.

\subsubsection{Duality}
Given a line bundle $\cL$ on $\Ell_\bT(\bX)$, we set
\begin{equation}
\cL^\dd = \cL^{-1} \otimes \Theta(T\bX) \,. \label{dd}
\end{equation}
This is the appropriate twist of duality for line bundles in our
situation. 
If $p_\bX: X \to \pt$ is proper, this gives a pairing
\begin{equation}
  \label{eq:12}
  p_{\bX,*}( \cL \otimes \cL^\dd )  \xrightarrow{\quad p_{\bX,\ppf} \quad}
  \cO_{\cE_\bT} \,. 
\end{equation}

\subsubsection{Supports}

For any coherent sheaf $\cG$ on  $\Ell_\bT(\bX)$ and any
$\bT$-invariant open $U \subset \bX$, we have a subsheaf
\begin{equation}
  \label{Gsupp}
  \cG_{\supp \subset \bX \setminus U} = \Ker \left( \cG \to \cG|_{U}
  \right) 
\end{equation}
of sections supported in $\bX\setminus U$. The map in \eqref{Gsupp} is
the functorial pullback with respect to the inclusion $U \to \bX$.

We
will be using \eqref{Gsupp} for locally free sheaves $\cG$. 

\subsubsection{}
For general complex oriented maps, we have the following
generalization of \eqref{fpf}
\begin{equation}
f\pf \in \Hom_{\Coh \Ell_{\bG_2}(\bX_2)}
(f_* \Theta_{\bX_1}(-N_f)\cs,\cO_{\Ell_{\bG_2}(\bX_2)}) \label{fpf}
\,, 
\end{equation}
where the subsheaf
$$
\Theta_{\bX_1}(-N_f)\cs\subset \Theta_{\bX_1}(-N_f)
$$
is formed by sections $s$ such that $f\big|_{\supp s}$ is
proper.

\subsubsection{Correspondences}

Consider the diagram
\begin{equation}\label{diagcorr} 
  \xymatrix{&\bX_2 \times \bX_1 \ar[dl]_{\forp_1} \ar[dr]^{\forp_2} \\
    \bX_2 \ar[dr]_{p_2}&& \bX_1  \ar[dl]^{p_1}\\
    & \pt 
  }
\end{equation}
in which
$$
\forp_1 = 1_{\bX_2} \times p_1 \,, \quad \forp_2 = p_2 \times 1_{\bX_1} \,.
$$
In particular
$$
N_{\forp_1} = -  \forp_2^* \, T \bX_1\,.
$$
The pull-push formula
\begin{equation}
  \label{pullpush}
  \alpha \mapsto (\forp_2)\pf ( \alpha \forp_1^*(\, \cdot \,))  \,, \qquad
  \end{equation}
gives a map 
\begin{equation}
p_{12,*} \Theta_{\bX_1 \times \bX_2}(\forp_2^* \,  T\bX_1)\cs \to
\cHom\Big(
p_{1,*}  \cO_{\Ell_\bT(\bX_1)} , p_{2,*}  \cO_{\Ell_\bT(\bX_2)}
\Big)\label{pullpush2}
\end{equation}
of sheaves on
$\cE_\bT$. Here
$$
p_{12} = p_1 \circ \forp_2 = p_2 \circ \forp_1 \,.
$$

\subsubsection{}

More generally, if $\cL_i \in \Pic(\Ell_\bT(\bX_i))$ then a section
\begin{equation}
\alpha \in H^0(\cL_2 \boxtimes \cL_1^\dd)  \label{boxtimes}
\end{equation}
gives a map
\begin{equation}
p_{1,*}  \, \cL_1  \xrightarrow{\quad  (\forp_1)\pf ( \alpha \forp_2^*(\, \cdot
  \,))  \quad}  p_{2,*}  \, \cL_2 \label{pullpushL}
\end{equation}
in $\Coh \cE_\bT$. In \eqref{boxtimes} and elsewhere, the symbol
$\boxtimes$ denotes the tensor product of pullbacks via two projection
maps.

\subsubsection{}
Elliptic stable envelopes are objects of the form
\eqref{boxtimes} on $\bX \times \bX^{\bA}$ for certain special lines
bundles $\cL_i$.

\subsection{K\"ahler line bundles} \label{s_Kah}

\subsubsection{} 

K\"ahler line bundles give a rich supply of line bundles of
$\bA$-degree $0$. The terminology is chosen for consistency with \cite{ese}, because the variables
$z$ corresponds to the K\"ahler parameters in enumerative
context.

\subsubsection{} 

Let $V$ be vector bundle on $\bX$ and let $\Ct_z$ be a new copy of the
group $\Ct$ with coordinate $z$, acting trivially on $\bX$. As in
\eqref{pullA}, we have 
\begin{equation}
  \label{eq:3}
  \Ell_{\bT \times \Ct_z}(\bX) = \Ell_{\bT}(\bX) \times E_z\,, \quad
  E_z=\cE_{\Ct_z}\,. 
\end{equation}
Consider the K-theory class
$$
(z-1)(V-\C^{\rk V}) \in \fmo^2 \subset K_{\bT \times
  \Ct_z}(\bX)\,, 
$$
where $\C^{\rk V}$ is a trivial bundle of rank $r$.
It gives a line bundle 
\begin{equation}
\cU(V,z) = \Theta((z-1)(V-\C^{\rk V})) \in \Pic(\Ell_{\bT}(\bX) \times E_z)
\, , \label{defcU}
\end{equation}
which has degree $0$ along each factor.

\subsubsection{}

Clearly, 
$$
\cU(V_1 \oplus V_2,z) \cong \cU(V_1,z)  \otimes \cU(V_2,z)\,. 
$$

\begin{Lemma}
  We have
  \begin{align}
    \label{euVz1}
    \cU(V,z) &\cong \cU(\det V,z)\,, \\
    \cU(V,z_1z_2) &\cong \cU(V,z_1)
    \otimes \cU(V,z_2)\,. \label{euVz2}
  \end{align}
\end{Lemma}

\noindent
Here $\cU(V,z_1z_2)$ is the pullback by the addition map $E \times E \to
E$, as in \eqref{groupl}.

\begin{proof}
  Follows immediately from Lemma \ref{l_cube} \,. 
\end{proof}

\subsubsection{}

By \eqref{euVz1} and homotopy invariance, the map $V\mapsto \cU(V,z)$
factors through the discrete group
$$
K_\bT(\bX) \xrightarrow{\quad \det \quad} \Pic_{\bT}(\bX)
\xrightarrow{\quad \pi_0 \quad}
NS_\bT(\bX) \overset{\tiny\textup{def}} = \Pic_{\bT}(\bX) \big/
\Pic_{\bT}(\bX)_0 \,. 
$$

\subsubsection{}

If $V$ pulled back via $f: \bX \to Y$, then similarly $\cU(V,z)$ is
pulled back from $\Ell_\bT(Y)$. In particular, line bundles $\cU(w, z)$, where $w$ is a character of
$\bT$, are pulled back from $\cE_\bT$.

\section{Line bundles on $\cE_\bA$}

\subsection{Proof of Lemma \ref{l_cube}} \label{s_A1}

\begin{proof}
  It suffices to consider the universal case
  $$
  \bG = \prod GL(r_i)\,, \quad  \bX=\pt \,. 
  $$
  To save on notation, we set
  $$
  \cE = \Ell_\bG(\pt) = \prod S^{r_i} E \,. 
  $$
  This is smooth and projective over $\bB$. 
  
  Consider the diagonal embedding and three projections 
  $$
  \bG \xrightarrow{\quad\diag\quad}  \bG^3
  \xrightarrow{\quad p_1,p_2,p_3\quad}   \bG \,. 
$$
The ideal $\fmo^3$ is spanned by $V$ of the form
$$
V = \diag^* \left( p_1^*(V_1) \otimes  p_2^*(V_2) \otimes p_3(V_3)
\right) \,, \quad V_i \in \fmo \,. 
$$
Since $V_i \in \fmo$, we have $V_i\big|_{1\in G} =0$.  Therefore
$$
\left. \Theta\left(\bigotimes  p_i^*(V_i)  \right)\right|_{\{1\} \times
  \cE
  \times \cE} = \cO_{\cE^2}  \,, 
$$
and similarly for the other factors.

By the theorem of the cube, see \cite{stacks} for a version over a general
base scheme $\bB$, we have
$$
\Theta\left(\bigotimes  p_i^*(V_i)  \right) = \cO_{\cE^3} \,.
$$
Therefore $\Theta(V) = \diag^* \cO_{\cE^3} = \cO_\cE$, as was to be
shown.
\end{proof}

\subsection{The dual Abelian variety} 

Since the character lattice $\cha(\bA)$ is dual to the
cocharacter lattice in \eqref{cEbT} and $E$ is principally polarized,
we have 
$$
\cE^\vee_\bA = \cha(\bA) \otimes_\Z E  = \cE_{\bA^\vee}\,,
$$
where $\bA^\vee$ is the dual torus.

If $\{a_i\}$ is a basis of $\cha(\bA)$ and $\{a^\vee_i\}$ is the dual
  basis of characters of $\bA^\vee$ then
  $$
  \eta = \sum (a_i-1)(a^\vee_i-1)  \in \fmo^2/\fmo^3 \subset K_{\bA \times
    \bA^\vee}(\pt)/\fmo^3 
  $$
  is a canonical element and
  $$
  \Theta(\eta) = \textup{Poincar\'e line bundle on $\cE_{\bA} \times
      \cE_{\bA^\vee}$}
  \,.
  $$

  \subsection{The degree of a line bundle} \label{A_degree}

  By definition, we have an exact sequence
  $$
  0 \to \Pic_0(\cE)  \to \Pic(\cE) \to \NS(\cE) \to 0\,, 
  $$
  where the discrete group $\NS(\cE)$ is  the Neron-Severi group of
  $\cE$.  For an Abelian variety, we have
  $$
  \Pic_0(\cE) = \cE^\vee\,, \quad \NS(\cE) = \Hom_{\textup{symmetric}}(\cE,\cE^\vee)  \,.
  $$
By definition, the degree $\cL$ of a line bundle is its class in
  $\NS(\cE)$.  To compute it, note that 
  $\cL$ defines a morphism $\cE_\bA \to
  \cE_{\bA^\vee}$ by 
  by the formula
  $$
  a \mapsto (\textup{multiplication by $a$})^* \cL \otimes \cL^{-1}
\,. 
$$ 
We denote this morphism by
$$
\deg \cL \in \Hom_{\textup{symmetric}}(\cE_\bA, \cE_\bA^\vee) \cong
\NS(\cE) \,. 
$$

{} From definitions, if
$$
V = \sum_{\mu \in \cha(\bA)} V_\mu \, a^\mu \in K_{\bA}(\bX^\bA)
$$
then 
\begin{equation}
\deg \Theta(V)
= \sum \rk(V_{\mu}) \, \mu^2 \in S^2 \cha(\bA) \subset
\Hom_{\textup{symmetric}}(\cE_\bA, \cE_\bA^\vee) \,.\label{degThV2}
\end{equation}
Note that for a very general elliptic curve $E$, we have
$$
\Hom(E,E) = \Z \,,
$$
and therefore $\NS(\cE) =  S^2\cha(\bA)$. We will not encounter line
bundles of other degrees in this paper.

\begin{figure}[!h]
  \centering
   \includegraphics[scale=0.4]{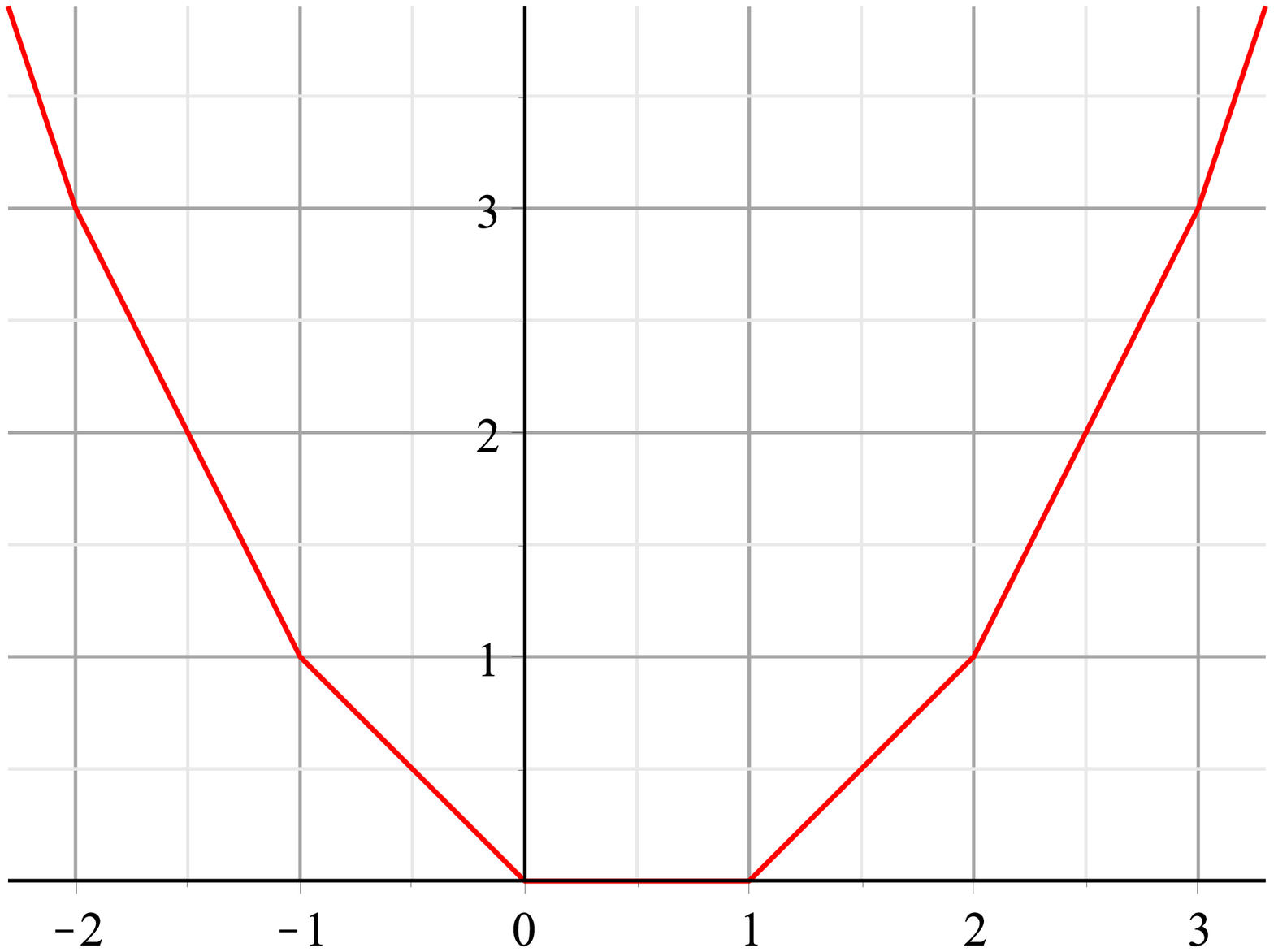}
 \caption{The function $\Qq(x)$}
\label{f1}
\end{figure}

\section{Degeneration of Abelian varieties}

Degeneration of Abelian varieties is a very rich, old, and well-developed
subject, see \cite{Brion} for an introduction to its basics including
Alexeev's theory 
\cite{Alex}. Here we summarize a few elementary aspects of this theory
that will play a role in this paper.

\subsection{Tate elliptic curve}\label{s_Tate}

Let $\Qq(x)$ be a solution of
$$
\Qq''(x) = \sum_{n\in \Z} \delta(x-n) \,,
$$
which vanishes for $x\in [0,1]$. Explicitly,
$$
\Qq(x) = x \flx - \tfrac12 \flx (\flx+1)
$$
see the plot in Figure \ref{f1}. We note that changing $\Qq(x)$  by an integral
form of degree $\le 1$ in $x$  does not affect what
follows, while the effect of changing it by a rational form of degree
$\le 1$ will be discussed in Section \ref{s_fraction}.

Let $\cR$ be a ring with unit. The Tate elliptic curve
over $\cR[[q]]$ is the quotient of the toric surface $\Sigma_\Qq$ over
$\cR$ with polyhedron 
\begin{align}
\Delta_\Qq &= \{(x,y), y \ge \Qq(x)\}\notag \\
  &= \conv\left\{ (k,m) \in \Z^2, m \ge \frac12 \,
    k(k-1)\right\} \label{DelQq} 
\end{align}
by a certain action of $\Lambda \cong \Z$. The completion of $\cR[q]$
to $\cR[[q]]$ will be required to take the quotient.

By definition, 
\begin{equation}
\Sigma_\Qq = \Proj \bigoplus_{\frac{m}{n} \ge
  \Qq\left(\frac{k}{n}\right)}
\cR \, t^n a^k q^m \,, \quad
n\ge 0 \,, \label{Sigmaproj}
\end{equation}
where the $n=0$ term is interpreted as $\cR[q]
$ and the
grading is by $n$. The fact that the algebra in \eqref{Sigmaproj} is
not finitely generated should not a be a concern, as the invariants of
$\Lambda$ will be finitely generated.

The group $\Lambda$ is generated by the
transformation\footnote{Up to a rescaling of the
coordinates, the group $\Lambda$ is the translation part in the affine
Weyl group of type $\widehat{A}_1$. We note that while hyperplane
arrangement appearing in the theory of stable envelopes may be
considered
as a generalization of the notion of roots and coroots in classical
Lie theory, the reflection group aspect is lost in this generalization.}
\begin{equation}
\sigma: (t,a,q)\mapsto (at,qa,q) \,. \label{actsig}
\end{equation}
In other words,
$$
\Lambda = \left\langle {\textstyle
    \begin{pmatrix}
      1 & 0 & 0\\
      1 & 1 & 0\\
      0 & 1 & 1
    \end{pmatrix}}
  \right\rangle
\subset SL(3,\Z)\,, 
$$
acting on triples $(n,k,m)\in \Z^3$. The group $\Lambda$ preserves the
cone over $\Delta_\Qq$ by virtue of \eqref{DelQq} because it preserves
the lattice and the quadratic form $2mn+kn-k^2$. 

In conventional terms, the action \eqref{actsig} is the action by
$q$-difference operators on functions of $a$. In particular, familiar
$\vth$-functions like
\begin{equation}
\vth_0(a,q) = \sum_{k\in \Z} a^k q^{\Qq(k)} \in \cR[a^{\pm
  1}][[q]]\label{theta_ser}
\end{equation}
are well defined and satisfy
\begin{equation}
\sigma( t \vth_0(c a,q)) =  c^{-1} \, t \vth_0(c a,q) \,,\label{thca}
\end{equation}
assuming $c$ is invertible and invariant under $\sigma$.

To continue the parallel with the classical theory of theta-functions,
one may characterize the functions
\begin{equation}
  \label{Fn}
  \cF_n= \widehat{\bigoplus_{\frac{m}{n} \ge
  \Qq\left(\frac{k}{n}\right)}}
\cR \, a^k q^m  \subset \cR[a^{\pm 1}]\Fq \,, 
\end{equation}
appearing in \eqref{Sigmaproj}, in terms of their norms at various
points
$$
a_0: \cR[a^{\pm 1}]\Fq \xrightarrow{\qquad}  \cR'\FqN \,, \quad N=1,2,\dots
\,,
$$
extending a ring homomorphism $\cR \to \cR'$. 
As usual, we will denote by $f(a_0,q)$ the image of $f(a,q)$ under the
corresponding map. The norm will be measured by the valuation 
\begin{equation}
\cR'\FqN \xrightarrow{\quad\bnu\quad} \tfrac 1N \Z \subset
\R\,,\label{val}
\end{equation}
given by the minimal nonzero power of $q$. In principle, one can also
consider evaluations of $a$ to irrational powers of $q$, but this does
not add
anything new.

\begin{Lemma}
  Functions in $\cF_n$ are those satisfying the
  following norm bound: 
\begin{equation}
\cF_n = \left\{f \, \Big| \, \forall a_0,
  \bnu(f(a_0,q))+n \, \Qq \left(\bnu (a_0)\right) \ge 0\right\}
\,.\label{Fn2}
\end{equation}
\end{Lemma}
\begin{proof}
  Follows from the observation that
  the function $\Qq(x)$ is its own Legendre transform up to a
flip of sign
\begin{equation}
\Qq^\vee(\alpha) = \max_{x}  \left( \alpha x - \Qq(x) \right) =
\Qq(-\alpha) \,.\label{QqLeg}
\end{equation}
\end{proof}

\noindent 
{}It is obvious from a description of the form \eqref{Fn2} that 
$$
\cF_0 = \cR[[q]] \,.
$$
and that 
$$
\cF_n \cF_m \subset \cF_{n+m} \,. 
$$
The group $\Lambda$ acts on each $\cF_n$ by $q$-difference
operators. These operators cover the action $a\mapsto q a$ and exist because
\begin{equation}
\Qq(x+1)=\Qq(x) + x \,.\label{Qqdiff}
\end{equation}
By definition, 
\begin{equation}
E_\textup{Tate} = \Proj \left(\bigoplus_{n \ge 0} \cF_n^\Lambda
  \right)
  \,.\label{Etate} 
\end{equation}
\begin{Lemma}
The algebra in \eqref{Etate} finitely generated over $\cR[[q]]$. 
\end{Lemma}

\begin{proof}
  It suffices to prove finite generation modulo $q$. Setting $s=ta$ in
  \eqref{Sigmaproj} this amounts to proving finite generation of
  the following subalgebra
  $$
  \cR[t,s] \supset \left\{ f(t,s) \, \big|\, f(t,0)=f(0,t) \right\}\,, 
  $$
  which is the homogeneous coordinate ring of $\bP^1_\cR$ with $0$ and
  $\infty$ identified. One sees directly that it is generated over $\cR$ by
  $$
    f_1 = t+s\,, \quad 
    f_2 = ts \,, \quad
    f_3 = t^2s \,, 
    $$
    which satisfy the relation
    $$
    f_3^2 + f_2^3 - f_1 f_2 f_3 = 0\,, 
    $$
  in which one can recognize formula (46) in \cite{Tate}. 
\end{proof}

\subsection{A higher rank generalization}\label{s_cQ}

\subsubsection{}

Let let $\bA\cong \Gm^r$ be a split torus of some rank $r$ and let
$\{\mu_i\}$ be a collection of weight of $\bA$ spanning the vector
space $\fa^*$ in
\eqref{deffa}. 
Interpreting each $\mu_i$
as a linear function on $\fa$, we define
\begin{equation}
\cQ = \sum \Qq(\mu_i) \,.\label{cQ}
\end{equation}
This is a nondegenerate convex function on $\fa$.

As our running
example, we will take $r=2$ and
\begin{equation}
\cQ(x,y) = \Qq(2x) + \Qq(y) + \Qq(x-y) \,.\label{cQex}
\end{equation}
This function is plotted in Figure \ref{f2} together with its Legendre
transform. 
\begin{figure}[!h]
  \centering
  \includegraphics[scale=0.4]{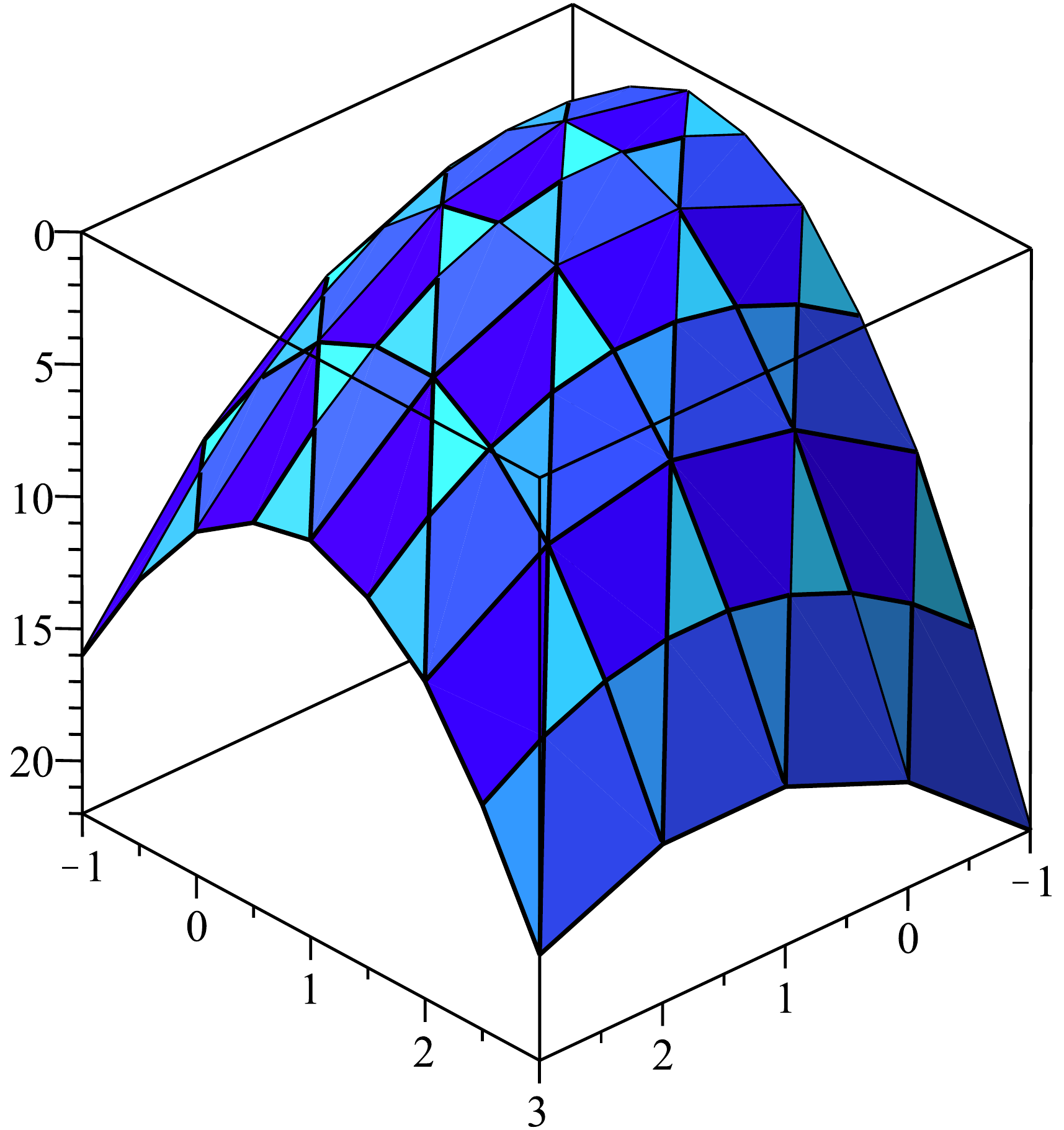}
  \hspace{0.5cm}\raisebox{-0.35cm}{\includegraphics[scale=0.54]{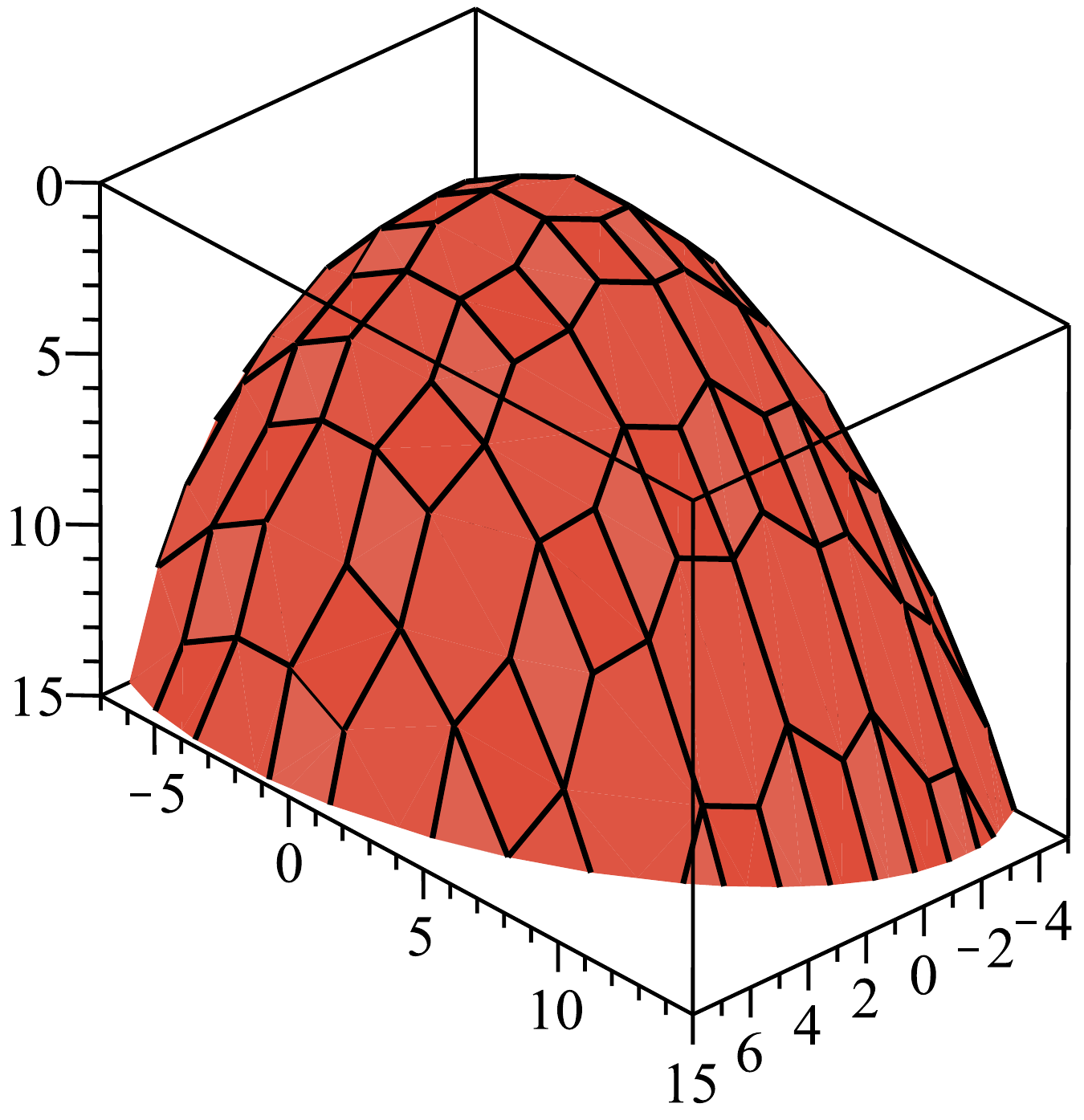}}
 \caption{The function $\cQ$ in \eqref{cQex} and its Legendre
   transform $\cQ^\vee$, with the $z$-axis is upside down. Note the
   duality between faces of the periodic tilings of $\fa$ and $\fa^*$ obtained by
   the projection to the argument plane. The tiling of $\fa$ is by
   hyperplanes $\mu_i(x) \in \Z$.}
\label{f2}
\end{figure}

We define $\cF_{\cQ,n} \subset \cR[\bA]\Fq$ by 
\begin{equation}
\cF_{\cQ,n}= \left\{f \, \Big| \, \forall a_0,
  \bnu(f(a_0,q))+n \, \cQ \left(\bnu (a_0)\right) \ge 0\right\}
\,.\label{FQn}
\end{equation}
Here the point $a_0$ has $r$ coordinates and so $\bnu(a_0)\in \fa$.

\subsubsection{}

Generalizing \eqref{Qqdiff}, for any cocharacter 
$$
\sigma \in \cochar\bA \subset \fa
$$
we have
\begin{equation}
\cQ(x+\sigma) = \cQ(x) + \sigma^\vee(x) + \cQ(\sigma)\label{cQdiff}
\end{equation}
where the self-adjoint map
\begin{equation}
\cochar(\bA) \owns \sigma \, \mapsto \sigma^\vee \in
\cha(\bA)\label{sisidu}
\end{equation}
is induced by the integral quadratic form $\sum \mu_i^2$. In our
running example, it has the the matrix $
\begin{pmatrix}
5 & -1 \\ -1 & 2
\end{pmatrix}
$.

Because of \eqref{cQdiff}, there is an action of
$$
\Lambda= \cochar(\bA)
$$ on
$\cF_{\cQ,n}$ by $q$-difference operators 
lifting its action by
$$
(a,q) \xrightarrow{\quad \sigma \quad} (\sigma(q) a, q) \,. 
$$ 
It induces an action on the toric $(r+1)$-fold defined by the Legendre
transform of $\cQ$. In particular, the tiling on the left in Figure
\ref{f2} is $\Z^2$-periodic, while the tiling on the right is
invariant under $\begin{pmatrix}
5 & -1 \\ -1 & 2
\end{pmatrix} \, \Z^2$.

\subsubsection{}\label{logexp}

Over $\C$, there is an obvious correspondence between algebraic
subgroups $\Gamma\subset \bA$ and Lie subgroups of its compact form
$\fa/\Lambda$. In general, it can be rephrased as follows. 

Recall that the equations of a subgroup
$\Gamma\subset \bA$  are encoded by the
kernel in the following exact
sequence
$$
0 \to \Gamma^\perp \to \chr(\bA) \to \chr(\Gamma) \to 0 \,. 
$$
We define
\begin{equation}
\log \Gamma = \{ x \in \fa \, | \, \lambda(x) \in \Z, \forall \lambda
\in \Gamma^\perp \}\,, \label{logGamma} 
\end{equation}
This is a Lie subgroup of $\fa$ and a 
periodic locally finite arrangement of affine subspaces of
$\fa$.
We also define the subgroup $\exp(\eta) \subset \bA$ by 
\begin{equation}
(\exp \eta)^\perp  = \{ \lambda \in \chr(\bA) \, | \, \lambda(\eta)
\in \Z
\}\,. \label{expeta} 
\end{equation}
We will be using \eqref{expeta} in the situation when $\eta$ is not
necessarily a subgroup, in particular, when it a stratum in a
stratification of $\fa$. In this case, $\log \exp \eta$ is the closed subgroup
of $\fa/\Lambda$ generated by $\eta$. 

\subsubsection{}\label{s_Gammaeta}

Consider 
\begin{equation}
\Proj \bigoplus_{n \ge 0} \cF_{\cQ,n}^{\Lambda}
\,.\label{cEcQdef0}
\end{equation}
This is scheme over the spectrum of  $\cF_{\cQ,0}=\cR[[q]]$.

With $q$
inverted, \eqref{cEcQdef0} is
an Abelian variety, namely $\Ell_\bA(\pt)$ for the 
elliptic curve $E_\textup{Tate}$, also with $q$ inverted. The $q=0$
fiber is a union of toric varieties corresponding to orbits of
$\Lambda$ on the tiling of $\fa^*$.

Note that, in general, the torus $\bA$ action on the 
toric varieties in the $q=0$ fiber of \eqref{cEcQdef0} is not faithful. For instance, the quadrilaterals in
the tiling on the right in Figure \ref{f2} are dual to the points
\begin{equation}
  \{(\tfrac12,0), (\tfrac12, \tfrac12) \} +\Z^2 \label{slopes12}
  \subset \fa_\Q 
\end{equation}
in the hyperplane arrangements on the left in the same figure. This 
means that the corresponding facets have the slopes
\eqref{slopes12},
and hence
intersect the lattice $\Z^3$ in lattices that projects to sublattices
of index 2 in $\chr \bA \cong \Z^2$.

More generally, let $\eta  \subset \fa$ be a stratum of our hyperplane
arrangement, let $\eta^\vee$ be the dual stratum of $\fa^*$
and let $O_{\eta^\vee}$ be the corresponding torus orbit in the
$q=0$ fiber of \eqref{cEcQdef0}. We have $O_{\eta^\vee} \cong \bA
/\exp(\eta)$, while we would like to have 
\begin{equation}
  \label{eq:6}
  O_{\eta^\vee} \cong \bA/ \exp(d \eta)\,,  
\end{equation}
where $d\eta \subset \fa$ is the tangent space to $\eta$ and
$\exp(d \eta) \subset \bA$ is the corresponding subtorus. Note that
$\exp(d \eta)$ 
is the reduced component of
the identity of $\exp(\eta)$. 

To achieve this, we do a base change of the form $q=(q^{1/M})^M$
where $M$ is so divisible that all facets of $\cQ^\vee$ intersect the
lattice $\chr \bA \oplus \frac1M \Z$ in the same lattice as their
projection to $\chr \bA$. We define
\begin{equation}
\cE_{\cQ} = \Proj \bigoplus_{n \ge 0} \cF_{\cQ,n}^{\Lambda}[q^{1/M}] 
\,.\label{cEcQdef}
\end{equation}
For instance, in our running example,
it suffices to introduce $q^{1/2}$.

\subsection{Fractional shifts}\label{s_fraction}

As a projective spectrum, $\cE_\cQ$ is constructed together with a
distinguished ample line bundle $\cO(\cQ)$. We now want to
geometrically realize shifts of $\cO(\cQ)$ by points of finite order
in the Picard group of the generic fiber. 

The generic fiber $\cE_{\cQ,\gen}$ of $\cE_{\cQ}$ is the product of $r$ elliptic curves
and hence, principally polarized. This 
gives an identification of the points of
finite order
$$
\Pic_0 (\cE_{\cQ,\gen})_\tor \cong (\cE_{\cQ,\gen})_\tor \,. 
$$
We have, see \cite{Tate}, 
$$
0 \to \mu_\infty \to (E_{\textup{Tate},\gen})_\tor  \xrightarrow{\quad \bnu
  \quad}  \Q/\Z \to 0 
$$
where the quotient is a constant group scheme over the spectrum of of
$\cR$. In our situation, this sequence is canonically split by
fractional powers of $q$.

Correspondingly
\begin{equation}
\Pic_0 (\cE_{\cQ,\gen})_\tor  = \fa^*_\Q/ \cha(\bA) \oplus
\Hom(\Lambda, \mu_\infty) \,.\label{PicEllA}
\end{equation}
The second term in \eqref{PicEllA} is realized by taking spectrum of
the corresponding covariant algebra in \eqref{cEcQdef}. In concrete
terms, this amounts to taking $c\in \mu_\infty$ in \eqref{thca}.

\begin{Lemma}
The first term in \eqref{PicEllA} is realized by
$$
\cQ \mapsto \cQ + \lambda \,,  \quad \lambda \in \fa^*_\Q \,.
$$
\end{Lemma}

\noindent
Note that this shifts all tiles in $\fa^*$ by $\lambda$. If $\Delta
\subset \fa$ is one such tile, then it replaces the line bundle
$\cO(\Delta)$ on the corresponding toric variety by the sheaf 
$\cO(\Delta_\lambda)$ as in Section \ref{s_Dellam}. In particular,
if $\lambda$ is integral this doesn't change $\cE_\cQ$ and shifts
$\cO(\cQ)$ by a character. 

\begin{proof}
Let $N$ denote the order of $\lambda$ in $\fa^*_\Q/\cha(\bA)$ as in
Section \ref{s_orderN}.  The function
$$
\cQ_\lambda= \cQ+\lambda
$$
is integral for
the torus $\bA_N$ in \eqref{bAN} hence the formula \eqref{cEcQdef}
produces a scheme $\cE_{\cQ,N}$ with $\mu_N$-equivariant line bundle
such
that 
\begin{equation*}
      (\cE_{\cQ_\lambda}, \cO(\cQ_\lambda)) =  (\cE_{\cQ,N}, \cO(\cQ) \, a^\lambda)
      \Big/ \!\!\!\! \Big/\mu_N \,. 
\end{equation*}
Over the generic fiber, this is an isogeny and realizes the required
twist via
$$
\cha(\mu_N) = \Z/N\Z \xrightarrow{\quad \lambda \quad}  \fa^*_\Q/
\cha(\bA) \,.
$$
\end{proof}

In more familiar terms, the above proposition says that one
can take the shift $c$ in \eqref{thca} to lie in $\mu_\infty
q^{\Q}$. Indeed, shifting the variables by a fractional power of $q$
has the effect of adding a fractional linear function to $\cQ$.

\end{document}